\documentclass[11pt,a4paper,twoside,english]{elsarticle}
\usepackage[latin9]{inputenc}
\usepackage[active]{srcltx}
\usepackage{xcolor}
\definecolor{document_fontcolor}{rgb}{0, 0, 0}
\color{document_fontcolor}
\usepackage{babel}
\usepackage{mathrsfs}
\usepackage{amsmath}
\usepackage{amsthm}
\usepackage{amssymb}
\usepackage{esint}
\PassOptionsToPackage{normalem}{ulem}
\usepackage{ulem}
\usepackage[unicode=true,
 bookmarks=false,
 breaklinks=false,pdfborder={0 0 0},backref=false,colorlinks=true]
 {hyperref}
\usepackage{breakurl}

\makeatletter

\providecolor{lyxadded}{rgb}{0,0,1}
\providecolor{lyxdeleted}{rgb}{1,0,0}

\usepackage{enumitem}		
\newlength{\lyxlistindent}      
\theoremstyle{plain}
\newtheorem{thm}{\protect\theoremname}
  \theoremstyle{definition}
  \newtheorem{defn}[thm]{\protect\definitionname}
  \theoremstyle{remark}
  \newtheorem{rem}[thm]{\protect\remarkname}
  \theoremstyle{plain}
  \newtheorem{lem}[thm]{\protect\lemmaname}
  \theoremstyle{plain}
  \newtheorem{cor}[thm]{\protect\corollaryname}
  \theoremstyle{plain}
  \newtheorem*{thm*}{\protect\theoremname}

\usepackage{pdfsync}
\usepackage{graphics}
\usepackage{epstopdf}
\usepackage[all]{xy}
\usepackage{amscd}
 \usepackage{ipa}
                         
\newcommand{\xyR}[1]{ \makeatletter
\xydef@\xymatrixrowsep@{#1} \makeatother} 
\newcommand{\xyC}[1]{ \makeatletter
\xydef@\xymatrixcolsep@{#1} \makeatother} 
\entrymodifiers={++[ ][F]} 

\newcommand{\ra}{\longrightarrow}
\newcommand{\xrad}[1]{\xrightarrow{\displaystyle\ \ #1\ \ }} 

\newcommand{\field}[1]{\mathbb{#1}}
\newcommand{\R}{\field{R}} 
\newcommand{\N}{\field{N}} 

\renewcommand{\phi}{\varphi} 
\newcommand{\diff}[1]{\,\hbox{\rm d}#1} 

\newcommand{\Coo}{\mbox{\ensuremath{\mathcal{C}}}^{\infty}} 
\DeclareMathOperator*{\Man}{{\bf Man}} 
\newcommand{\FDiff}{\ext{\Coo}} 
\newcommand{\SFR}{{\bf S}\mbox{$\FR^\infty$}} 
\newcommand{\Diff}{{\bf Dlg}} 

\newcommand{\FR}{{{}^\bullet\R}} 
\newcommand{\st}[1]{{\makebox[0pt][r]{\phantom{$#1$}}^{\circ} #1}} 
\newcommand{\sh}[1]{{\makebox[0pt][r]{\phantom{$#1$}}^{\text{\textnormal{\nisigma}}} #1}} 
\newcommand{\ext}[1]{{\makebox[0pt][r]{\phantom{$#1$}}^{\bullet} #1}} 
\newcommand{\extFR}{{}^\bullet\bar{\R}} 
\def\XXint#1#2#3{{\setbox0=\hbox{$#1{#2#3}{\int}$}
\vcenter{\hbox{$#2#3$}}\kern-.5\wd0}}
\def\Rint{\mathchoice
{\XXint\displaystyle\textstyle{\sim}}
{\XXint\textstyle\scriptstyle{\sim}}
{\XXint\scriptstyle\scriptscriptstyle{\sim}}
{\XXint\scriptscriptstyle\scriptscriptstyle{\sim}}
\!\int} 

\newcommand{\ptind}{\displaystyle \mathop {\ldots\ldots\,}} 
\newcommand{\then}{\quad \Longrightarrow \quad} 

\makeatother

  \providecommand{\corollaryname}{Corollary}
  \providecommand{\definitionname}{Definition}
  \providecommand{\lemmaname}{Lemma}
  \providecommand{\remarkname}{Remark}
  \providecommand{\theoremname}{Theorem}
\providecommand{\theoremname}{Theorem}

\begin{document}

\begin{frontmatter}{}

\title{Calculus in the ring of Fermat reals\\
{\large{}Part I: Integral calculus}}

\author[PG]{Paolo Giordano\corref{corr}\fnref{PGfn}}

\ead{paolo.giordano@univie.ac.at}

\author[EW]{Enxin Wu\fnref{EWfn}}

\ead{enxin.wu@univie.ac.at}

\cortext[corr]{Corresponding author}

\fntext[PGfn]{P. Giordano has been supported by grants P25116 and P25311 of the
Austrian Science Fund FWF.}

\fntext[EWfn]{E. Wu has been supported by grants P25116 and P25311 of the Austrian
Science Fund FWF.}

\address[PG]{University of Vienna, Austria}

\address[EW]{University of Vienna, Austria}
\begin{abstract}
We develop the integral calculus for quasi-standard smooth functions
defined on the ring of Fermat reals. The approach is by proving the
existence and uniqueness of primitives. Besides the classical integral
formulas, we show the flexibility of the Cartesian closed framework
of Fermat spaces to deal with infinite dimensional integral operators.
The total order relation between scalars permits to prove several
classical order properties of these integrals and to study multiple
integrals on Peano-Jordan-like integration domains.\end{abstract}

\end{frontmatter}{}

\tableofcontents{}

\section{\label{sec:Introduction}Introduction}

In spite of old controversies concerning the existence and rigor of
infinitesimals, it is remarkable that using only elementary calculus,
one can define and deeply study non-trivial rings containing the real
field and infinitesimals. This goal has been achieved e.g.\ with
the Levi-Civita field (\citep{Lev1,Lev2,Berz2,Berz-et-al,Sha,Sh-Be}),
the Colombeau ring of generalized numbers (\citep{Col84,Col2,Obe92})
or the ring of Fermat reals (\citep{Gio10a,Gio10b,Gio10e,Gi-Ku13,Gio09}).

A common erroneous opinion is that in non-Archimedean theories, we
pay the price of needing a non-trivial knowledge of mathematical logic
(\citep{Hen}) or switching to intuitionistic logic (\citep{Koc})
or using some non-trivial set theoretical methods (\citep{Con2}).
The non-Archimedean theories mentioned above need none of these. Another
common feeling is that, to deal with infinitesimals one needs to have
a great formal attitude in doing mathematics. Of course, there are
half-way approaches trying to require a less formal path: e.g.\ at
least up to a certain level, also nonstandard analysis, surely the
most powerful theory of infinitesimals and infinities, can be presented
without the need to face formal logic (\citep{Lux73}).

Aside from theories based on analysis, several algebraic methods are
also available (\citep{Tall,Berz2,Kr-Mi,Ber}). However, in those
cases, the interested reader has to be able either to proceed formally,
or to develop an intuition related to such algebraic infinitesimals.
For example, in \citep{Gio10b} it is argued that an intuitive picture
of infinitesimal segments of lengths $h$ and $k$ such that $h^{2}=k^{2}=0$
but $h\cdot k\ne0$ is not possible. These numbers are frequently
used in Synthetic Differential Geometry (SDG, \citep{Koc}). At the
end, it is a matter of taste about what approaches are felt as beautiful,
manageable and in accordance with our philosophical approach to mathematics.

The ring $\FR$ of Fermat reals, can be defined and studied using
only elementary calculus; see (\citep{Gio10a}). It extends the field
$\R$ of real numbers, and contains nilpotent infinitesimals, i.e.,\ $h\in\FR_{\ne0}$
such that $h^{n}=0$ for some $n\in\N_{>1}$. The methodological thread
followed in the development of the theory of Fermat reals has been
always the necessity to obtain a good dialogue between formal properties
and their informal interpretations. Indeed, to cite some results of
this dialogue, we can say that the ring $\FR$ is totally ordered
and geometrically representable (\citep{Gio10b}).

Related to the ring $\FR$, we have a good notion of smooth function
called\emph{ }quasi-standard smooth function. These functions are
arrows of the category $\FDiff$ of Fermat spaces (\citep{Gio10e,Gio09}),
which is essentially a generalization of the notion of diffeological
space (\citep{Igl}), but starting from quasi-standard smooth functions
defined on arbitrary subsets $S\subseteq\FR^{\sf s}$ instead of smooth
functions defined on open sets $U\subseteq\R^{\sf u}$. The category
$\FDiff$ is a quasi-topos \citep{GioWu15b}, and hence has several
desired categorical properties, such as Cartesian closedness. Furthermore,
$\FDiff$ contains finite-dimensional smooth manifolds, and infinite
dimensional spaces like $\Man(M,N)$, the space of all smooth functions
between two smooth manifolds $M$ and $N$, and integral and differential
operators, etc. One of our goals is to develop differential geometry
on these spaces using nilpotent infinitesimals. For the first steps,
we need to develop differential and integral calculus, which is also
preliminary to other advanced topics, such as calculus of variations,
generalized functions, stochastic infinitesimals, etc.

As a start, in the present work, we develop integral calculus for
quasi-standard smooth functions. Later in another paper, we will develop
differential calculus. This exchange is motivated by the need of using
Taylor's theorem with integral remainder in differential calculus.

In our approach, the derivatives are defined using the existence and
uniqueness of the quasi-standard incremental ratio (see the Fermat-Reyes
Thm.~\ref{thm:Fermat-Reyes-h-small}), and the integrals are defined
using the existence and uniqueness of primitives (see Thm.~\ref{thm:existenceOfIntegralForFiniteBoundaries}
below). Following this approach, we have an extension of the classical
results of calculus and also a quick agreement with several informal
calculations used in physics (see e.g.\ \citep{Gio10b}).

What are the new results following this approach, compared to classical
analysis\footnote{Here by classical analysis, we mean the most developed theory of infinite
dimensional spaces, i.e.,\ locally convex topological vector space
theory. We recall that the category of these spaces is not Cartesian
closed; see \citep{Gio10c}.} or NSA? Besides the possibility to use nilpotent infinitesimals to
simplify several calculations, we will see that the Cartesian closed
framework $\FDiff$ of Fermat spaces permits to easily prove that
all the classical infinite dimensional integral and differential operators
are smooth maps, i.e., they are morphisms in $\FDiff$. On the other
hand, convenient vector spaces and related smooth functions (\citep{Fr-Kr,Kr-Mi})
can be embedded both in the category $\Diff$ of diffeological spaces
and in the category $\FDiff$ of Fermat spaces (see \citep{Gio10e,Igl}).
Therefore, one of the best calculus in locally convex topological
vector spaces is still available in $\FDiff$. Finally, the Fermat
functor $\ext{(-)}:\Diff\ra\FDiff$ (\citep{Gio09,GioWu15b}) connects
these categories and has very good preservation properties.

What are the new results compared to smooth infinitesimal analysis
(SIA; see e.g.\ \citep{Lav})? In the ring of Fermat reals, we have
a total order relation rather than only a pre-order, and indeed we
are able to prove classical inequalities and order properties of integrals.
Because of the simplicity of the whole model (if compared to those
of SIA), in our framework we can not only formally repeat several
proofs of SIA, but also suitably generalize some classical existence
theorems like the inverse function theorem and the intermediate value
theorem.

The structure of the paper is as follows. 

In Section \ref{sec:Background-for-FR}, we summarize the basics of
Fermat reals: unique decomposition of elements, Taylor's formula with
nilpotent increments, Fermat topology, extension of ordinary smooth
functions, well-ordering, quasi-standard smooth functions, etc.

In Section \ref{sec:A-uniqueness-theorem}, we review another topology
on Fermat reals called the $\omega$-topology (\citep{Gi-Ku13}),
which is only used in this section. Unlike the Fermat topology, we
show that in general, quasi-standard smooth functions are not continuous
with respect to this topology. On the other hand, we prove a uniqueness
theorem (Thm. \ref{thm:uniqueness}), which says that every quasi-standard
smooth function is determined by its values on any dense subset of
the domain in the $\omega$-topology. This uniqueness theorem is a
generalization of previous results like the cancellation law of non-infinitesimal
functions (\citep[Lem. 25]{Gio10e}), and it will be useful for subsequent
sections.

In Section \ref{sec:The-constant-function}, we show that every interval
and its interior in $\FR$ is connected (Lem. \ref{lem:intervalsConnected}),
and we use this to prove the constant function theorem (Thm. \ref{thm:constancyPrinciple}),
i.e., that only constant quasi-standard smooth functions have zero
derivatives on a non-infinitesimal interval (see Sec. \ref{sec:Background-for-FR}).
From now on, we require the ambient quasi-standard smooth function
to be defined on a non-infinitesimal interval, due to the uniqueness
of the incremental ratio in the Fermat-Reyes theorem (Thm. \ref{thm:Fermat-Reyes-h-small}).
Since the domain of every quasi-standard smooth function defined on
an interval of $\FR$ can always be extended to a Fermat open set
(although the extension may not be unique; see Lem. \ref{lem:extensionFcnOutsideBoundaries}),
we generalize the Fermat-Reyes theorem for quasi-standard smooth functions
defined on a non-infinitesimal interval instead of a Fermat open set,
and hence introduce the notion of left and right derivatives (Lem. \ref{lem:FermatReyesClosedInterval}).

In Section \ref{sec:Existence-and-uniqueness}, we prove the existence
and uniqueness of primitives for quasi-standard smooth functions (Thm. \ref{thm:existenceOfIntegralForFiniteBoundaries}
and Thm. \ref{thm:existenceOfIntegralForInfiniteBoundaries}) by the
sheaf property of such functions and the constant function theorem
(Thm. \ref{thm:constancyPrinciple}). From this, we define integrals
of quasi-standard smooth functions (Def. \ref{def:integral}), which
is a natural generalization of integrals of ordinary smooth functions
(Rem. \ref{rem:afterDefIntegral}). Note that, although we require
the ambient quasi-standard smooth function to be defined on a non-infinitesimal
interval, the integrals can be defined on an infinitesimal interval.
From the proof of Thm. \ref{thm:existenceOfIntegralForFiniteBoundaries},
we get a useful expression of integrals (Lem. \ref{lem:IntegralAsfiniteSum})
as a sum of integrals of the local expression of quasi-standard smooth
functions as extension of ordinary smooth functions with parameters.

Although most results about integrals of quasi-standard smooth functions
can be proved using Lem. \ref{lem:IntegralAsfiniteSum}, the local
expression, (see Rem. \ref{rem:integral-local}), we prefer ``global''
proofs. In Section \ref{sec:Classical-formulas-of} and the first
part of Section \ref{sec:Infinite-dimensional-integral}, we prove
the classical formulas of integrals calculus (Thm. \ref{thm:classicalFormulasOfIntegralCalculus},
Thm. \ref{thm:integrationByChangeOfVariable}, Lem. \ref{lem:differentiationIntegral-x-Under}
and Cor. \ref{cor:integralAdditiveDomain}). The ``global'' proof
of the additivity of integrals with respect to the integration intervals
(Cor. \ref{cor:integralAdditiveDomain}) relies on the commutativity
of differentiation and integration (Lem. \ref{lem:derivationUnderIntegralSign}),
and the latter is a consequence of the $\FDiff$ smoothness of some
infinite dimensional integral operators (Lem. \ref{lem:smoothIntegralOperatorOn_0_1}
and Cor. \ref{cor:smoothnessOfIntegralOperators}). Finally, we get
a form of mean value theorem called Hadamard's lemma (Cor. \ref{cor:incrementalRatioIntegralFunction}).

In the rest of Section \ref{sec:Infinite-dimensional-integral}, we
study the standard and infinitesimal parts of an integral (Def. \ref{def:stdInfParts}
and Lem. \ref{lem:stdPartAndShadowOfFunction}), and we review divergence
and curl using suitable (nilpotent) infinitesimals in Fermat reals.
As a result, any integral over an infinitesimal integration interval
is infinitesimal, and any integral of an infinitesimal-valued quasi-standard
smooth function is infinitesimal.

Thanks to the total order on $\FR$, in Section \ref{sec:Inequalities-for-integrals},
we get some similar inequalities for our integrals as the classical
ones, with respect to arbitrary integration interval (monotonicity
of integral (Lem. \ref{lem:monotoneIntegral} and Thm. \ref{thm:monotoneIntegral}),
a replacement of the inequality for absolute value of quasi-standard
smooth functions (Thm. \ref{thm:substIntAbsValue}), and the Cauchy-Schwarz
inequality (Thm. \ref{thm:Cauchy-Schwarz})). 

In Section \ref{sec:Multiple-integrals-and}, we study multiple integrals
of quasi-standard smooth functions, with integration domain a finite
family of pairwise disjoint boxes (Thm. \ref{thm:intOnSumIntervals}
and Thm. \ref{thm:intElemSets}), and we prove Fubini's theorem (Thm. \ref{thm:Fubini}).

\section{\label{sec:Background-for-FR}Background for Fermat reals}

In this section, we summarize the basics of Fermat reals by listing
some properties that permit to characterize the ring $\FR$ (see \citep{Gio14a}
for a formal list of axioms). For a more detailed presentation and
for a description of the very simple model, see \citep{Gio10a,Gio09}.
See also \citep{Gi-Ku13} for the study of $\FR$ using metrics, the
characterization of its ideals, roots of nilpotent infinitesimals,
some applications to fractional derivatives, and a computer implementation.

At the end of this section, we fill in a gap in the proof of the Fermat-Reyes
theorem in \citep{Gio10e}, so that the statement of the theorem (see
Thm.~\ref{thm:Fermat-Reyes-h-small}) is slightly different, but
the differential calculus developed in \citep{Gio10e} still holds.

Fermat reals can be defined as a quotient ring of the ring of little-oh
polynomials. Equivalently, every Fermat real $x\in\FR$ can be written,
in a unique way, as
\begin{equation}
x=\st{x}+\sum_{i=1}^{N}\alpha_{i}\cdot\diff{t}_{a_{i}},\label{eq:decomposition}
\end{equation}
where $\st{x}$, $\alpha_{i}$, $a_{i}\in\R$ are standard reals,
$a_{1}>a_{2}>\dots>a_{N}\ge1$, $\alpha_{i}\ne0$, and $\diff{t}_{a}$
verifies the following properties:
\begin{align*}
\diff{t}_{a}\cdot\diff{t}_{b} & =\diff{t}_{\frac{ab}{a+b}}\\
\left(\diff{t}_{a}\right)^{p} & =\diff{t}_{\frac{a}{p}}\quad\forall p\in\R_{\ge1}\\
\diff{t}_{a} & =0\quad\forall a\in\R_{<1}.
\end{align*}
The expression \eqref{eq:decomposition} is called the \emph{decomposition}
of $x$, and the real number $\st{x}$ its \emph{standard part}. The
greatest number $a_{1}=:\omega(x)$ is called the \emph{order} of
$x$ and represents the greatest infinitesimal appearing in its decomposition.
When $x\in\R,$ i.e.,\ $x=\st{x}$, we set $\omega(x):=0$. We will
also use the notations $\omega_{i}(x):=a_{i}$ and $\st{x_{i}}:=\alpha_{i}$
for the $i$\emph{-th order} and the $i$\emph{-th standard part}
of $x$. The order $\omega(-)$ has the following expected properties:
\begin{align*}
\omega(x+y) & \le\max\{\omega(x),\omega(y)\}\\
\frac{1}{\omega(x\cdot y)} & =\frac{1}{\omega(x)}+\frac{1}{\omega(y)}\quad\text{if }\st{x}=\st{y}=0\text{ and }x\cdot y\neq0.
\end{align*}

Using the decomposition of $x$, it is not hard to prove that for
$k\in\N_{>1}$, $x^{k}=0$ iff $\omega(x)<k$. For $k\in\R_{\ge0}\cup\{\infty\}$,
the ideal
\[
D_{k}:=\left\{ x\in\FR\,|\,\st{x}=0,\ \omega(x)<k+1\right\} 
\]
plays a fundamental role in $k$-th order Taylor's formula with nilpotent
increments (so that the remainder is zero). Indeed, for $k\in\N_{\ge1}$
we have that $D_{k}=\left\{ x\in\FR\,|\,x^{k+1}=0\right\} $. We simply
write $D$ for $D_{1}.$ Every ordinary smooth function $f:A\ra\R$,
defined on an open set $A$ of $\R^{d}$, can be extended to a function
$\ext{f}:\ext{A}\ra\FR$ defined on
\[
\ext{A}:=\left\{ x\in\FR{}^{d}\,|\,\st{x}\in A\right\} ,
\]
preserving old values $f(x)\in\R$ for $x\in A$. Note that, for $A,B$
open sets of Euclidean spaces, we can naturally identify $\ext{(A\times B)}$
with $\ext{A}\times\ext{B}$, and hence there is no ambiguity to write
$\FR^{d}$ for $(\FR)^{d}$ or $\ext{(\R^{d})}$; see \citep[Thm. 19]{Gio10e}.
The mentioned Taylor's formula is therefore 
\begin{equation}
\forall h\in D_{k}^{d}:\ \ext{f}(x+h)=\sum_{\substack{j\in\N^{d}\\
|j|\le k
}
}\frac{h^{j}}{j!}\cdot\frac{\partial^{|j|}f}{\partial x^{j}}(x),\label{eq:TaylorStd}
\end{equation}
where $x\in A$, and $D_{k}^{d}=D_{k}\times\ptind^{d}\times D_{k}$.

There is a natural topology on $\FR^{d}$ consisting of sets of the
form $\ext{A}$ for open sets $A\subseteq\R^{d}$. We call this topology
the \emph{Fermat topology}, and open sets in the Fermat topology the
\emph{Fermat open sets}. Unless we specify otherwise, in the present
work we always equip every subset of $\FR^{d}$ with the Fermat topology
when viewed as a topological space.

It may seem difficult to work in a ring with zero divisors. But the
following properties permit to deal effectively with products of nilpotent
infinitesimals (typically appearing in Taylor's formula of several
variables) and with cancellation law:
\[
{\displaystyle {\displaystyle h_{1}^{i_{1}}\cdot\ldots\cdot h_{n}^{i_{n}}=0}\quad\iff\quad\sum_{k=1}^{n}\frac{i_{k}}{\omega(h_{k})}>1}
\]
\[
x\text{ is invertible}\quad\iff\quad\st{x}\ne0
\]
\[
\left(\text{If }x\cdot r=x\cdot s\text{ in }\FR,\text{ where }r,s\in\R\text{ and }x\ne0\right)\then r=s.
\]

The (commutative unital) ring $\FR$ is totally ordered, and the order
relation can be effectively decided starting from the decompositions
of elements as follows. For $x\in\FR$, if $\st{x}\ne0$, then $x>0\iff\st{x}>0$;
otherwise, if $\st{x}=0$, then $x>0\iff\st{x_{1}}>0$. 

For example, $\diff{t}_{3}-3\diff{t}>\diff{t}>0$, and $\diff{t_{a}}<\diff{t_{b}}$
if $a<b$. In the following, we will use symbols like $[a,b]:=\left\{ x\in\FR\mid a\le x\le b\right\} $
for intervals in $\FR$, whereas intervals in $\R$ will be denoted
by symbols like $[a,b]_{\R}:=[a,b]\cap\R$. By an\emph{ interval }in
$\FR$ we mean a set of the form $[a,b]$, $(a,b)$, $(a,b]$ or $[a,b)$,
where $a\leq b$. An interval is called\emph{ non-infinitesimal }if
$\st{a}<\st{b}$. Since $\leq$ is a total order on $\FR$, it is
easy to prove e.g. that $a=\inf(a,b)$, $b=\sup(a,b)$. Therefore,
each inteval of $\FR$ uniquely determines its endpoints.

The calculus we begin to develop in the present work concerns quasi-standard
smooth functions, which were introduced in \citep{Gio10e}. Recall
that a function $f:S\ra T$ with $S\subseteq\FR^{\sf s}$ and $T\subseteq\FR^{\sf t}$,
is called \emph{quasi-standard smooth} if locally (in the Fermat topology)
it can be written as 
\begin{equation}
f(x)=\ext{\alpha}(p,x)\quad\forall x\in\ext{V}\cap S,\label{eq:quasi-std-def}
\end{equation}
where $\alpha\in\Coo(P\times V,\R^{\sf t})$ is an ordinary smooth
function defined on an open set $P\times V\subseteq\R^{\sf p}\times\R^{\sf s}$,
and $p\in\ext{P}$ is some fixed $\sf p$-dimensional parameter\footnote{Note the different fonts used e.g.\ for the point $p\in P$ and the
dimension of $\R^{\sf p}\supseteq P$.}.\emph{ }Every quasi-standard smooth function is continuous when both
domain and codomain are equipped with the Fermat topology (see \citep[Thm. 13]{Gio10e}).
We write $\SFR(A,B)$ for the set of all quasi-standard smooth functions
from $A\subseteq\FR^{\sf a}$ to $B\subseteq\FR^{\sf b}$. Moreover,
when $A$ and $B$ are open sets of Euclidean spaces, $\SFR(A,B)=\FDiff(A,B)=\Coo(A,B)$
(see \citep[Thm. 18(4)]{Gio10e}). Therefore, the notions of \textquotedblleft ordinary
smooth function\textquotedblright{} (i.e., functions such that all
iterated partial derivatives exist and are continuous), of \textquotedblleft quasi-standard
smooth function\textquotedblright{} and of \textquotedblleft arrow
of the category $\FDiff$\textquotedblright{} (see \citep[Def. 16]{Gio10e})
are consistent. From now on, we call every such function a \emph{smooth
function}, and one can tell which kind it is from its domain and codomain.

In \citep{Gio10e}, the Fermat-Reyes theorem, which is essential for
the development of the differential calculus on $\FR$, was presented.
Unfortunately, there is a gap in formula (22) of \citep{Gio10e},
since nothing guarantees that the neighborhood $\ext{B}$ of $h$
is sufficiently small for that formula to hold. Nevertheless, it is
clear that we need to guarantee the validity of the Fermat-Reyes formula
(see \eqref{eq:FR} below) only for $h$ sufficiently small. In the
following definition, we precise the meaning of the sentence ``for
$h$ sufficiently small, the property $\mathcal{P}(h)$ holds''.
\begin{defn}
\label{def:for-h-suffSmall}Let $\mathcal{P}(h)$ be a property of
$h\in S\subseteq\FR$. We write 
\[
\forall^{0}h\in S:\ \mathcal{P}(h),
\]
and we read it as \emph{for $h\in S$ sufficiently small $\mathcal{P}(h)$
holds}, if
\[
\exists\rho\in\R_{>0}\,\forall h\in S:\left|h\right|<\rho\Rightarrow\mathcal{P}(h).
\]
Note explicitly that this notation also includes the special case
$S\subseteq\R$. Note also that, in a formula like $\forall x\,\forall^{0}h:\,\mathcal{P}(x,h)$,
if we say that $\mathcal{P}(x,h)$ holds for $|h|<\rho$, then $\rho$
depends on $x$.
\end{defn}
\noindent For example, if $U\subseteq\FR^{d}$ is a Fermat open set
and $v\in\FR^{d}$, then
\begin{equation}
\forall x\in U\,\forall^{0}h\in\FR:\ x+hv\in U.\label{eq:thickeningAndSmall-h}
\end{equation}
Using this language, the new statement of the Fermat-Reyes theorem
is the following. The above notation $\forall^{0}h$ is very convenient
whenever thickenings (see \citep[Lem.~23]{Gio10e}.) are implicit,
as in the following: 
\begin{thm}
\label{thm:Fermat-Reyes-h-small}Let $U$ be a Fermat open set of
$\FR$, and let $f:U\ra\FR$ be a smooth function. Then there exists
one and only one smooth map $r\in\FDiff(\widetilde{U},\FR)$ such
that
\begin{equation}
\forall(x,h)\in\widetilde{U}:\ f(x+h)=f(x)+h\cdot r(x,h)\quad\text{in}\quad\FR.\label{eq:FRwithThick}
\end{equation}
Without citing the thickening $\widetilde{U}$, we can state \eqref{eq:FRwithThick}
as
\begin{equation}
\forall x\in U\,\forall^{0}h\in\FR:\ f(x+h)=f(x)+h\cdot r(x,h)\quad\text{in}\quad\FR.\label{eq:FR}
\end{equation}
We can thus define $f^{\prime}(x):=r(x,0)\in\FR$ for every $x\in U$.
Moreover, if $f(x)=\ext{\alpha}(p,x)$, $\forall x\in\ext{V}\subseteq U$
with $\alpha\in\Coo(P\times V,\R)$, then
\end{thm}
\[
f'(x)={}^{^{^{{\scriptstyle \bullet}}}}{\!\left(\frac{\partial\alpha}{\partial x}\right)}(p,x)\quad\text{in}\quad\FR.
\]

\begin{proof}
Take $(x,h)\in\widetilde{U}$. We first consider the case $\st{h}=0$.
Since $x\in U$, we can write $f|_{\mathcal{V}_{x}}=\ext{\alpha_{x}}(p_{x},-)|_{\mathcal{V}_{x}}$,
where $\alpha_{x}\in\Coo(P_{x}\times\bar{V}_{x},\R)$, $\mathcal{V}_{x}:=\ext{\bar{V}_{x}}\subseteq U$
is a Fermat open neighborhood of $x$ and $\ext{P_{x}}$ is a Fermat
open neighborhood of some fixed parameter $p_{x}\in\FR^{{\sf p}_{x}}$.
Set $T_{x}:=\widetilde{\ext{\bar{V}}_{x}}$ for the thickening of
$\ext{\bar{V}}_{x}$. Since $T_{x}$ is a Fermat open subset of $\FR^{2}$
(see \citep[Lem.~23]{Gio10e}), we can find two open subsets $A$,
$B$ of $\R$ such that $\ext{(A\times B)}\subseteq T_{x}$, $\st{x}\in A$
and $\st{h}=0\in B$. Since $\widetilde{U}$ is also a Fermat open
set, without loss of generality, we can assume that $\ext{(A\times B)}\subseteq\widetilde{U}$.
By definition of thickening, we have $a+s\cdot b\in\bar{V}_{x}$ for
all $a\in A$, $b\in B$ and $s\in[0,1]_{\R}$. This proves correctly
property (22) of \citep{Gio10e}. From this point on, we can continue
like in \citep{Gio10e}, proving that we can define a smooth function
$r:\ext{(A\times B)}\ra\FR$ that satisfies \eqref{eq:FRwithThick}
in its domain. If $\st{h}\ne0$, then $h$ is invertible, and hence
we can simply define $r(x,h):=\frac{f(x+h)-f(x)}{h}$ in a neighborhood
$\ext{(A\times B)}\subseteq\widetilde{U}$ of $(x,h)$ such that $0\notin B$.
Using the sheaf property in $\FDiff$, exactly like in \citep{Gio10e},
we have the conclusion.
\end{proof}
Here is a simple application. If $g:\R\ra\R$ is a smooth function,
then $(\ext{g})'(x)=\ext{(g')}(x)$. The Fermat-Reyes theorem hence
implies that the function $\R^{2}\ra\R$ defined by

\[
(s,t)\mapsto\begin{cases}
\frac{g(s+t)-g(s)}{t}, & \text{if }t\neq0\\
g'(s), & \text{if }t=0
\end{cases}
\]
is smooth.

Note that the Fermat-Reyes theorem permits to differentiate smooth
functions only at interior points (in the Fermat topology).

\section{\label{sec:A-uniqueness-theorem}A uniqueness theorem}

Before discussing integration of quasi-standard smooth functions,
we prove in this section a very useful uniqueness theorem.

In order to state the theorem, we need another topology on Fermat
reals, called the $\omega$-topology. We summarize the basics of the
$\omega$-topology; see \citep{Gi-Ku13} for more details. 

Let $d_{\omega}:\FR^{n}\times\FR^{n}\ra\R_{\geq0}$ be defined by
$d_{\omega}(x,y)=\|\st{x}-\st{y}\|+\sum_{i=1}^{n}\omega(x_{i}-y_{i})$.
Then $d_{\omega}$ is a complete metric on $\FR^{n}$, and the topology
induced by this metric is called the \emph{$\omega$-topology},\emph{
}which is strictly finer than the Fermat topology. For any $x\in\FR^{n}$
and any $s\in(0,1]_{\R}$, every element in the open ball $B_{s}(x;d_{\omega})$
is of the form $x+r$ for $r\in\R^{n}$ with $\|r\|<s$ (see \citep[Thm. 7, 8, 11]{Gi-Ku13}).
We will denote by $A_{\omega}$ the topological space given by a subset
A of $\FR^{n}$ with the $\omega$-topology. Note that for $U$, $V$
open sets of Euclidean spaces, not every ordinary smooth function
$f:U\ra V$ induces a continuous function $\ext{f}:\ext{U}_{\omega}\ra\ext{V}_{\omega}$.
For example, if $f:\R\ra\R$ is the square function $f(x)=x^{2}$,
then $(\ext{f})^{-1}(B_{1}(\diff{t};d_{\omega}))=:X$ is not open
in $\FR_{\omega}$. In fact, one can check that $\diff{t_{2}}\in X$,
but no $\omega$-open neighborhood of it is contained in $X$.
\begin{rem}
\ 
\begin{enumerate}[leftmargin=*,label=(\roman*),align=left ]
\item The set of all invertible elements in $\FR$ is open and dense in
both $\FR_{\omega}$ and $\FR$, where the latter has the Fermat topology.
\item The ideal $D_{\infty}$ is closed in both $\FR_{\omega}$ and $\FR$.
More generally, every ideal in $\FR$ is closed in $\FR_{\omega}$,
but only the (unique) maximal (prime) ideal $D_{\infty}$ is closed
in $\FR$.
\item The sub-topologies on $\R^{n}$ via the embeddings $\R^{n}\ra\FR^{n}$
and $\R^{n}\ra\FR_{\omega}^{n}$ both coincide with the Euclidean
topology.
\end{enumerate}
\end{rem}
The following lemma is a trivial consequence of the uniqueness of
decomposition of Fermat reals:
\begin{lem}
\label{lem:quasi-decomposition}Assume that $x\in\FR$ can be written
as $x=r+\sum_{i=1}^{N}\alpha_{i}\cdot\diff{t}_{a_{i}}$, where $N\in\N$,
$r$, $\alpha_{1},a_{1},\ldots,\alpha_{N},a_{N}\in\R$ and $a_{1}>\ldots>a_{N}\ge1$.
Then $x=0$ implies $r=0$ and $\alpha_{i}=0$ for all $i=1,\ldots,N$.\end{lem}
\begin{proof}
Since $r=\st{x}=0$, we have the first part of the conclusion. If
$N=0$, the second part is trivial. Otherwise, $x=\sum_{\substack{i=1,\ldots,N\\
\alpha_{i}\ne0
}
}\alpha_{i}\cdot\diff{t}_{a_{i}}$ is the decomposition of $x$, and the conclusion then follows from
the uniqueness of the decomposition.
\end{proof}
\noindent A writing of the form $x=r+\sum_{i=1}^{N}\alpha_{i}\cdot\diff{t}_{a_{i}}$
satisfying the conditions of this lemma (compared to the decomposition
of $x$, we allow $\alpha_{i}$'s to be $0$) will be called a \emph{quasi-decomposition
of $x$.}

Here is the uniqueness theorem:
\begin{thm}
\label{thm:uniqueness} Let $A$ be an arbitrary subset of $\FR^{n}$,
and let $f,g:A\ra\FR$ be smooth functions. If $f(x)=g(x)$ for all
$x$ in a dense subset of $A_{\omega}$, then $f(x)=g(x)$ for all
$x\in A$.\end{thm}
\begin{proof}
We may assume that $g$ is the zero function. For any $a\in A$ there
exist an open neighborhood $V$ of $\st{a}$ in $\R^{n}$, an open
set $P\subseteq\R^{\sf p}$, a fixed parameter $p\in\ext{P}$, and
an ordinary smooth function $\alpha:P\times V\ra\R$ such that $f(x)=\ext{\alpha}(p,x)$
for all $x\in\ext{V}\cap A$. 

Let $p=(p_{1},\ldots,p_{\sf p})\in\FR^{\sf p}$ and $a=(a_{1},\ldots,a_{n})\in\FR^{n}$
be the components of $p$ and $a$. Let $p_{l}=\st{p_{l}}+\sum_{i=1}^{N_{l}}p_{il}\diff{t}_{q_{il}}=:\st{p_{l}}+k_{l}$,
$l=1,\ldots,\sf p$, and $a_{m}=\st{a_{m}}+\sum_{\iota=1}^{M_{m}}a_{\iota m}\diff{t}_{b_{\iota m}}=:\st{a_{m}}+h_{m}$,
$m=1,\ldots,n$, be the decompositions of these components with $k_{l}$,
$h_{m}$ their infinitesimal parts. For $r\in\R^{n}$ such that $\Vert r\Vert<1$
and $a+r\in\ext{V}\cap A$, we can write
\[
f(a+r)=\ext{\alpha}(p,a+r)=\ext{\alpha}(\st{p_{1}}+k_{1},\ldots,\st{p_{\sf p}}+k_{\sf p},\st{x_{1}}+h_{1},\ldots,\st{x_{n}}+h_{n}),
\]
where $x_{m}=\st{a_{m}}+r_{m}$ for $m=1,\ldots,n$. We show below
that the coefficients in a quasi-decomposition of $f(a+r)$ are ordinary
smooth functions of $r$. From the above decomposition, we have $k_{l}\in D_{q_{1l}}$
and $h_{m}\in D_{b_{1m}}$. Thereby, denoting by $\lceil u\rceil\in\N$
the ceiling of $u\in\R$ and setting $k:=(k_{1},\ldots,k_{\sf p})$,
$h:=(h_{1},\ldots,h_{n})$ and $c:=\max_{\substack{l=1,\ldots,\sf p\\
m=1,\ldots,n
}
}(\lceil q_{1l}\rceil,\lceil b_{1m}\rceil)$, we have that $(k,h)\in D_{c}^{{\sf p}+n}$. Using Taylor's formula
\eqref{eq:TaylorStd} for the smooth function $\alpha$ at the point
$(\st{p},\st{a}+r)$ and with increments $(k,h)$, we obtain
\begin{equation}
f(a+r)=\ext{\alpha}(p,a+r)=\sum_{j\in J}\partial_{j}\alpha(\st{p},\st{a}+r)\cdot\frac{(k,h)^{j}}{j!},\label{eq:Tay-r}
\end{equation}
where $J$ is the set of all the multi-indices $j\in\N^{{\sf p}+n}$
such that $|j|\le c$. Since we have that
\[
(k,h)^{j}=\prod_{l=1}^{\sf p}k_{l}^{j_{l}}\cdot\prod_{m=1}^{n}h_{m}^{j_{{\sf p}+m}}=\prod_{l=1}^{\sf p}\left(\sum_{i=1}^{N_{l}}p_{il}\diff{t}_{q_{il}}\right)^{j_{l}}\cdot\prod_{m=1}^{n}\left(\sum_{\iota=1}^{M_{m}}a_{\iota m}\diff{t}_{b_{\iota m}}\right)^{j_{{\sf p}+m}},
\]
i.e., a finite product of elements in the ideal $D_{c}$, the decomposition
of $(k,h)^{j}$ can be written as
\begin{equation}
(k,h)^{j}=\sum_{e=1}^{K_{j}}\gamma_{je}\cdot\diff{t}_{s_{je}}\label{eq:decPow}
\end{equation}
for some constants $K_{j}\in\N$, $\gamma_{je}\in\R$ and $s_{je}\in\R_{\geq1}$
(depending on the fixed datum $a$ and $f$, of course). For simplicity,
we can also use \eqref{eq:decPow} for $j=0\in\N^{{\sf p}+n}$ if
we set $K_{j}=\gamma_{je}=1$, $s_{je}=0$ and $\diff{t_{0}}=1$.
Substituting \eqref{eq:decPow} into \eqref{eq:Tay-r}, we get
\[
f(a+r)=\ext{\alpha}(p,a+r)=\sum_{j\in J}\sum_{e=1}^{K_{j}}\partial_{j}\alpha(\st{p},\st{a}+r)\cdot\frac{1}{j!}\cdot\gamma_{je}\cdot\diff{t_{s_{je}}}.
\]
To arrive at a quasi-decomposition of $f(a+r)$, it remains to gather
all the summands having the same term $\diff{t}_{s_{je}}$. Set
\[
L:=\left\{ s_{je}\mid j\in J\ ,\ e=1,\ldots,K_{j}\right\} ,
\]
let $\left\{ w_{1},\ldots,w_{H}\right\} :=L$ be an enumeration of
$L$ such that $w_{1}>\ldots>w_{H}$, and finally, for each $\ell=1,\ldots,H$,
set
\begin{multline*}
C_{\ell}(r):=\\
\sum\left\{ \partial_{j}\alpha(\st{p},\st{a}+r)\cdot\frac{1}{j!}\cdot\gamma_{je}\mid s_{je}=w_{\ell}\ ,\ j\in J\ ,\ e=1,\ldots,K_{j}\right\} .
\end{multline*}
Then
\begin{equation}
f(a+r)=\sum_{\ell=1}^{H}C_{\ell}(r)\cdot\diff{t}_{w_{\ell}}\label{eq:f(a+r)}
\end{equation}
is a quasi-decomposition of $f(a+r)$, and each $C_{\ell}(-)$ is
an ordinary smooth function of $r$ in a neighborhood of $0\in\R^{n}$.
The assumption implies that we can find a real sequence $(r(v))_{v\in\N}\downarrow0$
such that $a+r(v)\in\ext{V}\cap A$ and $f(a+r(v))=0$ for all $v\in\N$.
Therefore, $\sum_{\ell=1}^{H}C_{\ell}(r(v))\cdot\diff{t}_{w_{\ell}}=0$
for all $v$. Lem.~\ref{lem:quasi-decomposition} yields that $C_{\ell}(r(v))=0$
for all $\ell$, and hence $C_{\ell}(0)=0$ by the continuity of $C_{\ell}(-)$.
Setting $r=0$ in \eqref{eq:f(a+r)}, we get the conclusion that $f(a)=0$.
\end{proof}
This theorem will be used frequently in the following sections.

\section{\label{sec:The-constant-function}The constant function theorem}

Our first problem concerns the existence and uniqueness of primitives
of smooth functions $f:[a,b]\ra\FR$, and hence the starting of an
integration theory. We will solve this problem in next section, and
in this section, we first prove a useful instrument for solving this
problem, namely the constant function theorem. 

In the following, we use simplified symbols like $[a,b]$ also to
denote the corresponding subspace of $\FR$ in the category $\FDiff$,
i.e., for $\left([a,b]\prec\FR\right)$ (see \citep{Gio10e} for the
notation $(X\prec Y)$). We set $\extFR:=\FR\cup\{+\infty,-\infty\}$
with the usual operations, and $\st{a}=a$ if $a\in\{\pm\infty\}$,
even if we will never consider intervals closed at $-\infty$ or $+\infty$.

It is not hard to prove that if $a$, $b\in\R$, then
\begin{align*}
\ext{\left\{ (a,b)_{\R}\right\} } & \subsetneqq(a,b)\\
\ext{\left\{ [a,b]_{\R}\right\} } & \subsetneqq[a,b].
\end{align*}
For example, $x_{t}:=a-t^{2}$ is equal to $a$ in $\FR$, and hence
it belongs to the interval $[a,b]$, but $x\notin\ext{\left\{ [a,b]_{\R}\right\} }$
because $x$ does not map $\R_{\ge0}$ into $[a,b]_{\R}$. Analogously,
$a+\diff{t}\in(a,b)$ but $\st{(a+\diff{t})}=a\notin(a,b)_{\R}$,
so $a+\diff{t}\notin\ext{\left\{ (a,b)_{\R}\right\} }$. This implies
that the interval $(a,b)$, with $a$, $b\in\extFR$, is not open
in the Fermat topology if at least one of the endpoints is finite,
and thus we are initially forced to consider derivatives only at its
interior points, where
\begin{multline*}
\text{int}\left\{ \ext{\left\{ (a,b)_{\R}\right\} }\right\} =\text{int}\left\{ \ext{\left\{ [a,b]_{\R}\right\} }\right\} =\text{int}[a,b]=\text{int}(a,b)\\
=\{x\in\FR\mid\st{a}<\st{x}<\st{b}\}.
\end{multline*}
Here, of course, $\text{int}(A)$ denotes the interior of $A$ in
the Fermat topology, and we have used simplified notations like $\text{int}[a,b]=\text{int}\left\{ [a,b]\right\} $. 

As a second step, we will define the notion of right and left derivative
at $a$ and $b$, respectively, at the end of this section.

The following lemma is needed for the constant function theorem:
\begin{lem}
\label{lem:intervalsConnected}Let $J$ be an interval of $\FR$.
Then both the interval $J$ and its interior $\text{int}(J)$ are
connected in the Fermat topology.\end{lem}
\begin{proof}
Let $a$, $b$ be the endpoints of the interval $J$. We proceed for
$a$, $b$ finite and for the interval $(a,b)$, because the proof
is similar for the other cases. Assume that there exist two non-empty
relatively Fermat open sets $U_{1}$, $U_{2}$ of $(a,b)$ such that
$(a,b)=U_{1}\cup U_{2}$ and $U_{1}\cap U_{2}=\emptyset$. Set $I:=\{\st{x}\mid x\in(a,b)\}$.
It is not hard to see that $I=[\st{a},\st{b}]_{\R}\ne\emptyset$.
Moreover, setting $V_{i}:=\{\st{x}\mid x\in U_{i}\}$ for $i=1,2$,
we have that $V_{i}$ is a non-empty relatively open set of $I$,
$I=V_{1}\cup V_{2}$ and $V_{1}\cap V_{2}=\emptyset$, which implies
the conclusion since $I$ is connected in $\R$. The proof for the
case of the interior is similar since $\text{int}(a,b)=\{x\in\FR\mid\st{a}<\st{x}<\st{b}\}$.
\end{proof}
Here is the \emph{constant function theorem}, which is a stepstone
for the existence and uniqueness of primitives for smooth functions.
\begin{thm}
\noindent \label{thm:constancyPrinciple}Let $J$ be a non-infinitesimal
interval of $\FR$. Let $f:J\ra\FR$ be a smooth function such that
$f'(x)=0$ for each $x\in\text{int}(J)$. Then $f$ is constant on
$J$.\end{thm}
\begin{proof}
\noindent Let $a$, $b$ be the endpoints of the interval $J$. By
the uniqueness Thm.~\ref{thm:uniqueness}, it is enough to show that
$f$ is a constant function when restricted to $\text{int}(J)=\ext{\{(\st{a},\st{b})_{\R}\}}$,
and the assumption that $\st{a}<\st{b}$ (non-infinitesimal interval)
implies that $\text{int}(J)\neq\emptyset$. Let $x\in\text{int}(J)$,
so that
\begin{equation}
\forall^{0}h\in\FR:\ x+h\text{ is an interior point of }(a,b).\label{eq:interiorPoint}
\end{equation}
From the Fermat-Reyes Thm.~\ref{thm:Fermat-Reyes-h-small} we get
\begin{equation}
\forall^{0}h\in\FR:\ f(x+h)=f(x)+h\cdot f'[x,h].\label{eq:2_constancyPrinciple}
\end{equation}
Let us say that \eqref{eq:interiorPoint} and \eqref{eq:2_constancyPrinciple}
hold for $|h|<r\in\R_{>0}$ (here $r$ depends on $x$), and fix such
an $h$. Without loss of generality, we may assume that $|h_{t}|<r$
for each $t\ge0$, where $(h_{t})_{t}$ is a little-oh polynomial
representing $h$.

\noindent For simplicity, set $y:=f'[x,h]$. As proved in Thm.~\ref{thm:Fermat-Reyes-h-small}
(or Thm.~22 in \citep{Gio10e}), we can always find an ordinary smooth
function $\alpha$ and a parameter $p\in\FR^{\sf p}$ such that for
all $s\in[0,1]_{\R}$
\begin{align}
y_{t} & =\Rint_{0}^{1}\partial_{2}\alpha(p_{t},x_{t}+s\cdot h_{t})\,\diff{s}\quad\forall^{0}\,t\in\R_{\geq0}\text{ }\label{eq:1_constancyPrinciple}\\
f'(x+sh) & =\ext{(\partial_{2}\alpha)}(p,x+sh)\text{ in }\FR,
\end{align}
where $\Rint$ denotes the classical Riemann integral. As a motivation,
we can say that we would like to consider the function $y=\int_{0}^{1}\ext{(\partial_{2}\alpha)}(p,x+sh)ds$
with $p,x,h\in\FR$. For fixed $p,x,h\in\FR$, we understand it as
$y_{t}$ above, i.e., using little-oh polynomial representative for
each Fermat real, so that we get an ordinary integral. Note that this
integral is independent of the representatives we choose. But $|sh|<r$,
so that $x+sh$ is still an interior point, and hence by assumption
$f'(x+s\cdot h)=0$, i.e. $\ext{(\partial_{2}\alpha)}(p,x+s\cdot h)=0$
in $\FR$. Written in explicit form, this means
\[
\lim_{t\to0^{+}}\frac{\partial_{2}\alpha(p_{t},x_{t}+s\cdot h_{t})}{t}=0.
\]
From \eqref{eq:1_constancyPrinciple}, and using the Lebesgue dominated
convergence theorem, we have
\begin{align*}
\lim_{t\to0^{+}}\frac{y_{t}}{t} & =\lim_{t\to0^{+}}\frac{1}{t}\cdot\int_{0}^{1}\partial_{2}\alpha(p_{t},x_{t}+s\cdot h_{t})\,\diff{s}\\
 & =\int_{0}^{1}\lim_{t\to0^{+}}\frac{\partial_{2}\alpha(p_{t},x_{t}+s\cdot h_{t})}{t}\,\diff{s}=0.
\end{align*}
So $y=f'[x,h]=0$ in $\FR$, and hence \eqref{eq:2_constancyPrinciple}
yields $f(x+h)=f(x)$ for all $h$ such that $|h|<r$. This proves
that $f$ is locally constant in the Fermat topology, and from Lem.~\ref{lem:intervalsConnected}
the conclusion follows.
\end{proof}
We finally introduce the notion of left and right derivatives by means
of an extension of the Fermat-Reyes theorem from a Fermat open set
like $\text{int}(J)$ to the whole interval $J$. We first need the
following:
\begin{lem}
\label{lem:extensionFcnOutsideBoundaries}Let $a$, $b\in\extFR$
and let $J$ be a non-infinitesimal interval of $\FR$ with endpoints
$a$, $b$. Let $f:J\ra\FR$ be a smooth function. Then there exist
$\delta\in\R_{>0}$ and a smooth function $\bar{f}:(a-\delta,b+\delta)\ra\FR$
such that $\bar{f}|_{J}=f$.
\end{lem}
On the other hand, the smooth function $x\in\left\{ x\in\FR\mid0<\st{x}<1\right\} \mapsto\frac{1}{x}\in\FR$
cannot be extended to a smooth function defined on any interval containing
an infinitesimal.
\begin{proof}
By definition, we can write the function $f$ as a parameterized extension
of an ordinary smooth function both in a Fermat open neighborhood
of $a+\diff{t}$ and of $b-\diff{t}$, i.e.:
\begin{align}
f(x) & =\ext{\alpha}(p,x)\quad\forall x\in\ext{A}\cap J\label{eq:fIn-a}\\
f(x) & =\ext{\beta}(q,x)\quad\forall x\in\ext{B}\cap J,\label{eq:fIn-b}
\end{align}
where $A$, $B$ are open sets in $\R$ such that $\st{a}\in A$ and
$\st{b}\in B$, $\alpha\in\Coo(P\times A,\R)$, $\beta\in\Coo(Q\times B,\R)$,
$p\in\ext{P}\subseteq\FR^{\sf p}$, and $q\in\ext{Q}\subseteq\FR^{\sf q}$.
Therefore, any two of the smooth functions $\ext{\alpha}(p,-):\ext{A}\ra\R$,
$f|_{\text{int}(J)}$ and $\ext{\beta}(q,-):\ext{B}\ra\R$ are equal
in the intersection of their domains. Since $\st{a}\in A$ and $\st{b}\in B$,
we can always find some $\delta\in\R_{>0}$ such that $\{\ext{A},\text{int}(J),\ext{B}\}$
is a Fermat open covering of $(a-\delta,b+\delta)$ (let us note explicitly
that in this step we need to use the trichotomy law of the total order
$\le$ on $\FR$). The sheaf property of the $\FDiff$ space $(a-\delta,b+\delta)$
yields the existence of a smooth function $\bar{f}:(a-\delta,b+\delta)\ra\R$
such that $\bar{f}|_{J}=f|_{J}$.
\end{proof}
The notion of thickening $\widetilde{U}_{v}$ is defined only for
Fermat open sets $U$; see \citep{Gio10e}. A simple way to extend
this notion to arbitrary subsets is the following:
\begin{defn}
\label{def:thickeningClosedInt}Let $C\subseteq\FR^{d}$ and $v\in\FR^{d}$.
We say that $T\subseteq\FR^{d}\times\FR$ \emph{is a thickening of
$C$ along $v$} if there exist a Fermat open set $U\subseteq\FR^{d}$
such that $C\subseteq U$ and $T=\{(x,h)\in\widetilde{U}_{v}\mid x\in C\}$.
\end{defn}
Here is an extension of the Fermat-Reyes Thm.~\ref{thm:Fermat-Reyes-h-small}:
\begin{lem}
\label{lem:FermatReyesClosedInterval}Let $a$, $b\in\extFR$ and
let $J$ be a non-infinitesimal interval of $\FR$ with endpoints
$a$, $b$. Then the Fermat-Reyes theorem also holds if we replace
the Fermat open set $U$ with the interval $J$ in the statement of
Thm.~\ref{thm:Fermat-Reyes-h-small}. More precisely, if $x\in J$
and $x\simeq a$ (i.e., $\st{x}=\st{a}$), so that $x\in J\setminus\text{\emph{int}}(J)$,
then equation \eqref{eq:FR} shall be replaced by
\[
\forall^{0}h\in\FR_{>0}:\ f(x+h)=f(x)+h\cdot r(x,h)\quad\text{in}\quad\FR.
\]
Analogously, use $\forall^{0}h\in\FR_{<0}$ if $x\simeq b$.\end{lem}
\begin{rem}
Let $r:T\ra\FR$ be an incremental ratio of a smooth function $f:J\ra\FR$
defined on a non-infinitesimal interval $J$ of $\FR$, and let $s:S\ra\FR$
be an incremental ratio of the restriction $f|_{\text{int}(J)}$.
Then $s$ is essentially a restriction of $r$ in the sense that
\[
\forall x\in\text{int}(J)\,\forall^{0}h\in\FR:\ r(x,h)=s(x,h).
\]
We can hence define $f'(x):=r(x,0)$ for all $x\in J$, obtaining
an extension of $f'$ from $\text{int}(J)$ to the whole $J$.\end{rem}
\begin{proof}
Let $\bar{f}:(a-\delta,b+\delta)\ra\FR$ be an extension of $f:J\ra\FR$
as in Lem.\ \ref{lem:extensionFcnOutsideBoundaries}, set $U:=\text{int}(a-\delta,b+\delta)$
and let $\bar{f}'[-,-]:\widetilde{U}\ra\FR$ be the incremental ratio
of $\bar{f}$. Then $T:=\left\{ (x,h)\in\widetilde{U}\mid x\in J\right\} $
is a thickening of $J$ and $r:=\bar{f}'[-,-]|_{T}$ is a searched
incremental ratio of $f$. The uniqueness part follows from the uniqueness
Thm.~\ref{thm:uniqueness}.
\end{proof}

\section{\label{sec:Existence-and-uniqueness}Existence and uniqueness of
primitives}

We can now prove the existence and uniqueness of primitives for smooth
functions defined on an interval $[a,b]$ with finite endpoints.
\begin{thm}
\label{thm:existenceOfIntegralForFiniteBoundaries}Let $a$, $b\in\FR$
with $\st{a}<\st{b}$, let $f:[a,b]\ra\FR$ be a smooth function,
and let $u\in[a,b]$. Then there exists one and only one smooth function
\[
I:[a,b]\ra\FR
\]
such that
\[
I'(x)=f(x)\quad\forall x\in[a,b],\text{ {and}}
\]
\[
I(u)=0.
\]
\end{thm}
\begin{proof}
\noindent We may assume that $u=a$, since if $I'=f$ on $[a,b]$
and $I(a)=0$, then $J(x):=I(x)-I(u)$ verifies $J'=I'=f$ on $[a,b]$
and $J(u)=0$.

For every $x\in[a,b]$, we can write
\[
f|_{\mathcal{V}_{x}}=\ext{\alpha_{x}}(p_{x},-)|_{\mathcal{V}_{x}}
\]
for suitable $p_{x}\in\ext{U_{x}}\subseteq\FR^{{\sf p}_{x}}$, $U_{x}$
an open subset of $\R^{{\sf p}_{x}}$, $V_{x}$ an open subset of
$\R$ such that $x\in\ext{V_{x}}\cap[a,b]=:\mathcal{V}_{x}$ and $\alpha_{x}\in\Coo(U_{x}\times V_{x},\R)$.
We may assume that the open sets $V_{x}$ are of the form $V_{x}=(\st{x}-\delta_{x},\st{x}+\delta_{x})_{\R}$
for some $\delta_{x}\in\R_{>0}$.

The idea for the construction of the function $I$ is to patch together
suitable integrals of the functions $\ext{\alpha_{x}}(p_{x},-)$ and,
at the same time, to respect the condition $I(a)=0$. In order to
get a quasi-standard smooth function, we have to patch together integrals
in an order so that each one is the ``continuation'' of the previous
one. In other words, on the non-empty connected intersection of the
domains of any two of these integrals, the integrals must have the
same value at one point, so that we can prove that they are equal
on the whole intersection.

Since $\left(V_{x}\right)_{x\in[\st{a},\st{b}]_{\R}}$ is an open
cover of the compact set $[\st{a},\st{b}]_{\R}$, we can find a finite
subcover. That is, we can find $x_{1},\ldots,x_{n}\in[\st{a},\st{b}]_{\R}$
such that $\left(V_{x_{i}}\right)_{i=1,\ldots,n}$ is an open cover
of $[\st{a},\st{b}]_{\R}$. We will use simplified notations like
$V_{i}:=V_{x_{i}}$, $\delta_{i}:=\delta_{x_{i}}$, $\alpha_{i}:=\alpha_{x_{i}}$,
etc.

By shrinking each $V_{i}$ if necessary, we may always assume to have
chosen the indices $i=1,\ldots,n$ and the radii $\delta_{i}\in\R_{>0}$
such that
\[
\st{a}=x_{1}<x_{2}<\ldots<x_{n}=\st{b},\text{ {and}}
\]
\[
x_{i}-\delta_{i}<x_{i+1}-\delta_{i+1}<x_{i}+\delta_{i}<x_{i+1}+\delta_{i+1}\quad\forall i=1,\ldots,n-1.
\]
In this way, the intervals $V_{i}=(x_{i}-\delta_{i},x_{i}+\delta_{i})_{\R}$
and $V_{i+1}=(x_{i+1}-\delta_{i+1},x_{i+1}+\delta_{i+1})_{\R}$ intersect
in $(x_{i+1}-\delta_{i+1},x_{i}+\delta_{i})_{\R}$. For each $i=1,\ldots,n-1$,
we choose a point $\bar{x}_{i}\in(x_{i+1}-\delta_{i+1},x_{i}+\delta_{i})_{\R}$
and define, recursively:
\begin{align}
\iota_{1}(q,x,y): & =\Rint_{y}^{x}\alpha_{1}(q,s)\diff{s}\quad\forall q\in U_{x_{1}}\,\forall x,y\in V_{1},\label{eq:defI_1}\\
I_{1}(x): & =\ext{\iota_{1}}(p_{1},x,a)\quad\forall x\in\ext{V_{1}}.\nonumber 
\end{align}

\begin{align}
\iota_{i+1}(q,x): & =\Rint_{\bar{x}_{i}}^{x}\alpha_{i+1}(q,s)\diff{s}\quad\forall q\in U_{x_{i}}\,\forall x\in V_{i+1}\,\forall i=1,\ldots,n-1,\label{eq:defI_i+1}\\
I_{i+1}(x): & =\ext{\iota_{i+1}}(p_{i+1},x)+I_{i}(\bar{x}_{i})\quad\forall x\in\ext{V_{i+1}}\,\forall i=1,\ldots,n-1.\nonumber 
\end{align}
So every $I_{i}$ is a quasi-standard smooth function defined on $\ext{V_{i}}$,
and moreover, from Thm.~\ref{thm:Fermat-Reyes-h-small} it follows
that
\[
I'_{i}(x)=\ext{\alpha_{i}}(p_{i},x)=f(x)\quad\forall x\in\ext{V_{i}}.
\]
Therefore, for each point $x$ in 
\begin{align*}
\ext{V_{i}}\cap\ext{V_{i+1}} & =\ext{(V_{i}\cap V_{i+1})}=\ext{\{(x_{i+1}-\delta_{i+1},x_{i}+\delta_{i})_{\R}\}}\\
 & =\text{int}([x_{i+1}-\delta_{i+1},x_{i}+\delta_{i}]),
\end{align*}
we have
\begin{equation}
I'_{i}(x)=f(x)=I'_{i+1}(x),\text{ {and}}\label{eq:1_existenceOfIntegralForRealBoundaries}
\end{equation}
\[
\left(I_{i}-I_{i+1}\right)(\bar{x}_{i})=0.
\]
So, from the constant function Thm.~\ref{thm:constancyPrinciple}
it follows that $I_{i}=I_{i+1}$ on $[x_{i+1}-\delta_{i+1},x_{i}+\delta_{i}]$.
We can hence use the Fermat open cover $\left(\ext{V_{i}}\cap[a,b]\right)_{i=1,\ldots,n}$
of the space $[a,b]\in\FDiff$ to patch together the functions $I_{i}|_{\ext{V_{i}}\cap[a,b]}$
to obtain a smooth function $I:[a,b]\ra\FR$, by the sheaf property
of smooth functions. This function satisfies the required conditions
because of \eqref{eq:1_existenceOfIntegralForRealBoundaries} and
the equalities $I(a)=I_{1}(a)=0$.

To prove the uniqueness, suppose that $J$ verifies $J'=f$ on $[a,b]$
and $J(u)=0$. Then again by the constant function Thm.~\ref{thm:constancyPrinciple},
we know that $(J-I)|_{[a,b]}$ is constant and equal to zero at $u$,
and hence a zero function.
\end{proof}
The above theorem can be extended to smooth functions defined on an
arbitrary non-infinitesimal interval. 
\begin{thm}
\noindent \label{thm:existenceOfIntegralForInfiniteBoundaries}Let
$a$, $b\in\extFR$ and let $J$ be a non-infinitesimal interval of
$\FR$ with endpoints $a$, $b$. Let $f:J\ra\FR$ be a smooth function,
and let $u\in J$. Then there exists one and only one smooth function
$I:J\ra\FR$ such that $I'=f$ and $I(u)=0$.\end{thm}
\begin{proof}
\noindent We proceed for the interval $(-\infty,+\infty)$, since
the other cases are similar. Set $r:=\st{u}$, and for each $k\in\N_{>0}$,
define
\[
f_{k}:=f|_{[r-k,r+k]}.
\]
By Thm.~\ref{thm:existenceOfIntegralForFiniteBoundaries}, for each
$k$ there exists a unique smooth function $I_{k}:[r-k,r+k]\ra\FR$
such that $I_{k}'(x)=f_{k}(x)=f(x)$ for $x\in\text{int}[r-k,r+k]=:V_{k}$
and $I_{k}(u)=0$. Note that $\left(V_{k}\right)_{k>0}$ is a Fermat
open cover of $\FR$. Moreover, since $I_{k}'(x)=f(x)=I'_{j}(x)$
for every $x\in V_{k}\cap V_{j}$ and $I_{k}(u)=I_{j}(u)$, by the
constant function Thm.~\ref{thm:constancyPrinciple}, $I_{k}$ and
$I_{j}$ coincide on $V_{k}\cap V_{j}$. By the sheaf property of
smooth functions, we get
\begin{equation}
\exists!\,I\in\FDiff(\FR,\FR):\ I|_{V_{k}}=I_{k}\quad\forall k\in\N_{>0}.\label{eq:uniquenessI}
\end{equation}

\noindent Note that for any $x\in\FR$, there exists $k=k(x)\in\N_{>0}$
such that $x\in\text{int}(V_{k})$. Therefore, for each $x\in\FR$
and for $h\in\FR$ sufficiently small, we have
\[
I(x+h)=I_{k}(x+h)=I(x)+h\cdot I'[x,h]=I_{k}(x)+h\cdot I'_{k}[x,h].
\]
This and the uniqueness Thm.~\ref{thm:uniqueness} imply that the
incremental ratios of $I$ and $I_{k}$ are equal, i.e., $I'[x,h]=I'_{k}[x,h]$.
Thus, $I'(x)=I'_{k}(x)=f(x)$, and finally $I(u)=I_{1}(u)=0$.

This proves the existence part.

\noindent Since the restriction of a primitive of $f$ to $[r-k,r+k]$
is a primitive of $f_{k}$ for each $k$, the uniqueness then follows
from Thm.~\ref{thm:existenceOfIntegralForFiniteBoundaries}.
\end{proof}
We can now define
\begin{defn}
\label{def:integral}Let $J$ be a non-infinitesimal interval of $\FR$.
Let $f:J\ra\FR$ be a smooth function, and let $u\in J$. We define
$\int_{u}^{(-)}f$ to be the unique function verifying the following
properties:
\begin{enumerate}[leftmargin=*,label=(\roman*),align=left ]
\item ${\displaystyle \int_{u}^{(-)}f:=\int_{u}^{(-)}f(s)\,\diff{s}:J\ra\FR}$
is a smooth function;
\item ${\displaystyle \int_{u}^{u}f=0}$;
\item ${\displaystyle \forall x\in J:\ \left(\int_{u}^{(-)}f\right)'(x)=\frac{\diff{}}{\diff{x}}\int_{u}^{x}f(s)\,\diff{s}=f(x)}$.
\end{enumerate}
\end{defn}
\begin{rem}
\noindent \label{rem:afterDefIntegral}\end{rem}
\begin{enumerate}
\item Note that, although we require the interval $J$ to be non-infinitesimal
for the domain of the smooth function $f$, the integral $\int_{u}^{v}f$
is also defined when $v\in J$ and $v\simeq u$. 
\item It is not hard to see that if $f,g:J\ra\FR$ are smooth functions,
$a,b\in J$ with $a<b$, and $f|_{(a,b)}=g|_{(a,b)}$, then $\int_{a}^{b}f=\int_{a}^{b}g$.
We get the same conclusion if $\st{a}<\st{b}$ and we assume only
$f|_{\text{int}(a,b)}=g|_{\text{int}(a,b)}$.
\item \noindent From Def.~\ref{def:integral}, we obtain a generalization
of the usual notion of integral. Indeed, for $a$, $b$, $u\in\R$
with $a<u<b$, let $f\in\Coo([a,b]_{\R},\R)$. We extend $f$ to an
ordinary smooth function defined on an open interval $(a-\delta,b+\delta)_{\R}$,
for some $\delta\in\R_{>0}$. We still use the symbol $f$ to denote
this extension. We consider the quasi-standard smooth function
\[
I:={}^{^{^{^{{\scriptstyle \bullet}}}}}\left(\Rint_{u}^{(-)}f(s)\diff{s}\right):\ext{(a-\delta,b+\delta)_{\R}}\ra\FR.
\]
Now, we have that
\[
[a,b]\subseteq\ext{(a-\delta,b+\delta)_{\R}}
\]
so that we can consider the restriction $I|_{[a,b]}$. It is not hard
to prove that this restriction verifies all the properties of the
previous Def.~\ref{def:integral} for the function $\ext{f}$. Meanwhile,
because $I$ is the extension of an ordinary smooth function, it also
verifies
\begin{equation}
\forall x\in[a,b]_{\R}:\ \int_{u}^{x}\ext{f}(s)\diff{s}=I(x)={}^{^{^{^{{\scriptstyle \bullet}}}}}\left(\Rint_{u}^{x}f(s)\diff{s}\right)=\Rint_{u}^{x}f(s)\diff{s}\in\R.\label{eq:RiemannInt}
\end{equation}

\end{enumerate}
The recursive definitions \eqref{eq:defI_1} and \eqref{eq:defI_i+1}
applied to $f|_{[u,v]}$ yield the following useful result:
\begin{lem}
\label{lem:IntegralAsfiniteSum}Let $a$, $b\in\extFR$ and let $J$
be a non-infinitesimal interval of $\FR$ with endpoints $a$, $b$.
Let $f:J\ra\FR$ be a smooth function, and let $u$, $v\in J$ with
$u\leq v$. Then there exist $n\in\N_{>0}$, $y$, $\bar{y}$, $\delta\in\R^{n}$,
$(\alpha_{i})_{i=1}^{n}$, $(p_{i})_{i=1}^{n}$ such that
\begin{enumerate}[leftmargin=*,label=(\roman*),align=left ]
\item $\alpha_{i}\in\Coo(U_{i}\times V_{i},\R)$ and $p_{i}\in\ext{U_{i}}$,
where $U_{i}\subseteq\R^{{\sf p}_{i}}$, $V_{i}=(y_{i}-\delta_{i},y_{i}+\delta_{i})_{\R}$,
$\delta_{i}\in\R_{>0}$, $\forall i=1,\ldots,n$
\item $f(x)=\ext{\alpha_{i}}(p_{i},x)\quad\forall x\in\ext{V_{i}}\cap J\ \forall i=1,\ldots,n$
\item $\bar{y}_{i}\in(y_{i+1}-\delta_{i+1},y_{i}+\delta_{i})_{\R}=V_{i}\cap V_{i+1}\neq\emptyset\quad\forall i=1,\ldots,n-1$ 
\item $\st{u}=y_{1}<y_{2}<\ldots<y_{n}=\st{v}$ if $\st{u}<\st{v}$ and
$\st{u}=x_{1}=\st{v}$ otherwise
\item $\int_{u}^{v}f=\int_{u}^{\bar{y}_{1}}\ext{\alpha_{1}}(p_{1},s)\diff{s}+\int_{\bar{y}_{1}}^{\bar{y}_{2}}\ext{\alpha_{2}}(p_{2},s)\diff{s}+\ldots+\int_{\bar{y}_{n-1}}^{v}\ext{\alpha_{n}}(p_{n},s)\diff{s},$
and a little-oh polynomial representing $\int_{u}^{v}f$ is given
by
\[
t\in\R_{\ge0}\mapsto\Rint_{u_{t}}^{\bar{y}_{1}}\alpha_{1}((p_{1})_{t},s)\diff{s}+\ldots+\Rint_{\bar{y}_{n-1}}^{v_{t}}\alpha_{n}((p_{n})_{t},s)\diff{s}\in\R.
\]

\end{enumerate}
\end{lem}

\section{\label{sec:Classical-formulas-of}Classical formulas of integral
calculus}

In this section and the first part of next, we will prove the classical
formulas of integral calculus.

The first results concerning the integral calculus are the following:
\begin{thm}
\label{thm:classicalFormulasOfIntegralCalculus}Let $J$ be a non-infinitesimal
interval of $\FR$, let $f,g:J\ra\FR$ be smooth functions, and let
$u,$ $v\in J$. Then we have:
\begin{enumerate}[leftmargin=*,label=(\roman*),align=left ]
\item \label{enu:additivity}$\int_{u}^{v}\left(f+g\right)=\int_{u}^{v}f+\int_{u}^{v}g$
\item \label{enu:homog}$\int_{u}^{v}\lambda f=\lambda\int_{u}^{v}f\quad\forall\lambda\in\FR$
\item \label{enu:foundamental}$\int_{u}^{v}f'=f(v)-f(u)$
\item \label{enu:intByParts}$\int_{u}^{v}(f'\cdot g)=\left[f\cdot g\right]_{u}^{v}-\int_{u}^{v}(f\cdot g')$
\end{enumerate}
\end{thm}
We will prove the additivity with respect to the integration intervals,
i.e., $\int_{u}^{v}f+\int_{v}^{w}f=\int_{u}^{w}f$, after the proof
of the commutativity of differentiation and integration in next section.
The latter relies on the quasi-standard smoothness of the function
$x\mapsto\int_{0}^{1}f(s,x)\diff{s}$, which will be proved as a consequence
of the $\FDiff$ smoothness of suitable infinite dimensional integral
operators. 
\begin{rem}
\label{rem:integral-local}In fact, one can prove the above properties
together with the additivity with respect to the integration intervals
using Lem.~\ref{lem:IntegralAsfiniteSum}. As an example, we prove
in this remark that $\int_{u}^{v}f+\int_{v}^{w}f=\int_{u}^{w}f$ for
any smooth function $f$ defined on a non-infinitesimal inteval $J$
of $\FR$ with $u,v,w\in J$. By Lem.~\ref{lem:IntegralAsfiniteSum},
we can write 

\[
\int_{u}^{v}f=\int_{u}^{\bar{y}_{1}}\ext{\alpha_{1}}(p_{1},s)ds+\ldots+\int_{\bar{y}_{n-1}}^{v}\ext{\alpha_{n}}(p_{n},s)ds
\]

\[
\int_{v}^{w}f=\int_{v}^{\bar{z}_{1}}\ext{\beta_{1}}(q_{1},s)ds+\ldots+\int_{\bar{z}_{m-1}}^{w}\ext{\beta_{m}}(q_{m},s)ds.
\]
Since the function $f$ is quasi-standard smooth, we may assume that
$\alpha_{n}=\beta_{1}$, i.e., $\alpha_{n}(p_{1},x)=f(x)=\beta_{1}(q_{1},x)$
for every $x\in(\bar{y}_{n-1},\bar{z}_{1})$. The conclusion then
follows from the additivity of ordinary integral with respect to ordinary
integration intervals. 

This proof uses local expressions of quasi-standard smooth functions.
The proofs given below are of ``global'' nature, which is the point
we try to emphasize in the present work.\end{rem}
\begin{proof}
\noindent In the following, we will use the notations
\[
F:=\int_{u}^{(-)}f:J\ra\FR\quad,\quad G:=\int_{u}^{(-)}g:J\ra\FR.
\]
To prove \ref{enu:additivity}, it suffices to note that $F+G$ is
smooth and
\begin{align*}
\left(F+G\right)(u) & =F(u)+G(u)=0,\\
\left(F+G\right)'(x) & =F'(x)+G'(x)=f(x)+g(x)=\left(f+g\right)(x)\quad\forall x\in J.
\end{align*}
From the uniqueness of Def.~\ref{def:integral}, the conclusion $F+G=\int_{u}^{(-)}\left(f+g\right)$
follows.

For the proofs of the other properties, we can follow the same method:
we define a suitable quasi-standard smooth function $H$ starting
from the right hand side of the equality we need to prove, and we
prove that $H$ verifies the same properties that characterize uniquely
the function on the left hand side. For example, to prove \ref{enu:homog}
we consider $H(x):=\lambda F(x)$; to prove \ref{enu:foundamental}
we define $H(x):=f(x)-f(u)$; and finally to prove \ref{enu:intByParts}
we set $H(x):=f(x)\cdot g(x)-f(u)\cdot g(u)-\int_{u}^{x}(f\cdot g')$.
\end{proof}
Here is the result of \emph{integration by substitution}:
\begin{thm}
\label{thm:integrationByChangeOfVariable}Let $J$, $J_{1}$ be non-infinitesimal
intervals of $\FR$, let $u$, $v\in J_{1}$, and let
\[
J_{1}\xrad{\phi}J\xrad{f}\FR
\]
be smooth functions. Then
\[
\int_{\phi(u)}^{\phi(v)}f(t)\diff{t}=\int_{u}^{v}f\left[\phi(s)\right]\cdot\phi'(s)\diff{s}.
\]
\end{thm}
\begin{proof}
\textit{\emph{Define
\begin{align*}
F(x) & :=\int_{\phi(u)}^{x}f\quad\forall x\in J\\
H(y) & :=\int_{\phi(u)}^{\phi(y)}f\quad\forall y\in J_{1}\\
G(y) & :=\int_{u}^{y}f\left[\phi(s)\right]\cdot\phi'(s)\diff{s}\quad\forall y\in J_{1}.
\end{align*}
Each one of these functions is smooth because of Def.~\ref{def:integral}
(we recall that $u\in J_{1}$ is fixed) or because it can be written
as composition of smooth functions. We have $H(u)=G(u)=0$, $H(y)=F\left[\phi(y)\right]$
for every $y\in J_{1}$, and by the chain rule (\citep[Thm. 29]{Gio10e})
$H'(y)=F'[\phi(y)]\cdot\phi'(y)=f\left[\phi(y)\right]\cdot\phi'(y)=G'(y)$
for each $y\in J_{1}$. From the uniqueness of Thm.~\ref{thm:existenceOfIntegralForFiniteBoundaries}
the conclusion $H=G$ follows.}}
\end{proof}

\section{\label{sec:Infinite-dimensional-integral}Infinite dimensional integral
operators}

\noindent If we want to prove, in our framework, that
\begin{equation}
\frac{\diff{}}{\diff{x}}\left(\int_{0}^{1}f(s,x)\diff{s}\right)=\int_{0}^{1}\frac{\partial f}{\partial x}(s,x)\diff{s},\label{eq:diffIntSign}
\end{equation}
we need to prove that the function defined by
\begin{equation}
x\in[a,b]\mapsto\int_{0}^{1}f(s,x)\diff{s}\in\FR\label{eq:integrandSmooth}
\end{equation}
is smooth. Of course, in \eqref{eq:diffIntSign} $\frac{\partial f}{\partial x}(s,x):=\left[f(s,-)\right]'(x)$.
The smoothness of \eqref{eq:integrandSmooth} follows as a trivial
consequence of the smoothness of the function defined by
\begin{equation}
I_{a}^{b}:(f,u,v)\in\FDiff([a,b],\FR)\times[a,b]^{2}\mapsto\int_{u}^{v}f\in\FR.\label{eq:integralOperator}
\end{equation}
In fact, the function defined in \eqref{eq:integrandSmooth} can be
written as $I_{0}^{1}(-,0,1)\circ\left(f\circ\tau\right)^{\wedge}$,
where\footnote{\noindent See \citep[Def. 1]{Gio10e} for the notations $(-)^{\wedge}$
and $(-)^{\vee}$.}
\[
\tau:(x,s)\in[a,b]\times[0,1]\mapsto(s,x)\in[0,1]\times[a,b].
\]
Therefore, the smoothness of \eqref{eq:integrandSmooth} follows from
the smoothness of \eqref{eq:integralOperator} and the Cartesian closedness
of the category $\FDiff$ of Fermat spaces. For simplicity, in the
following we will denote by $\FR^{[a,b]}$ the space $\FDiff([a,b],\FR)$.

From Thm.~\ref{thm:integrationByChangeOfVariable}, we have
\[
\int_{u}^{v}f=\int_{0}^{1}f\left[y\cdot(v-u)+u\right]\cdot(v-u)\diff{y}.
\]
Thereby, setting
\begin{align*}
\psi(f,u,v):y & \in[0,1]\mapsto f\left[y\cdot(v-u)+u\right]\cdot(v-u)\in\FR\\
I:f & \in\FR^{[0,1]}\mapsto\int_{0}^{1}f\in\FR,
\end{align*}
it is easy to show that
\[
\psi:\FR^{[a,b]}\times[a,b]^{2}\ra\FR^{[0,1]}\text{ is smooth}
\]
and $I_{a}^{b}(f,u,v)=I\left[\psi(f,u,v)\right]$. Therefore, the
following result suffices to prove that the function $I_{a}^{b}$
is smooth.
\begin{lem}
\label{lem:smoothIntegralOperatorOn_0_1}The function
\[
I:f\in\FR^{[0,1]}\mapsto\int_{0}^{1}f\in\FR
\]
is smooth, i.e.,~it is an arrow of the category $\FDiff$.\end{lem}
\begin{proof}
Take a figure of $\FR^{[0,1]}$, i.e., let
\begin{align}
S & \subseteq\FR^{\sf s}\nonumber \\
d & \in\FDiff(S,\FR^{[0,1]}).\label{eq:figureOfExp}
\end{align}
Recall that here we use the simplified symbol $[0,1]$ to denote the
corresponding subspace of $\FR$, i.e.,\ $\left([0,1]\prec\FR\right)$.
From Property 5(i) of Thm.~18 in \citep{Gio10e} we have that $[0,1]=\left([0,1]\prec\FR\right)=\overline{[0,1]}$.
Therefore, from \eqref{eq:figureOfExp}, we know that $d^{\vee}:\bar{S}\times\overline{[0,1]}\ra\FR$
is smooth. We need to prove that the function
\begin{equation}
s\in S\mapsto\int_{0}^{1}d(s)=\int_{0}^{1}d^{\vee}(s,u)\diff{u}\in\FR\label{eq:thesis}
\end{equation}
is smooth, i.e.,\ that it belongs to $\SFR(S,\FR)$. From Properties
5(h) and 5(i) of Thm.~18 in \citep{Gio10e}, we have that
\begin{align*}
\bar{S}\times\overline{[0,1]} & =\left(S\prec\FR^{\sf s}\right)\times\left([0,1]\prec\FR\right)\\
 & =\left((S\times[0,1])\prec\FR^{\sf s+1}\right)\\
 & =\overline{S\times[0,1]}.
\end{align*}
Hence, $d^{\vee}:\overline{S\times[0,1]}\ra\FR$ is smooth, i.e.,
\begin{equation}
d^{\vee}:S\times[0,1]\ra\FR\text{ is an arrow of }\SFR,\label{eq:2_d_cohat}
\end{equation}
because $\SFR$ is fully embedded in the category $\FDiff$ of Fermat
spaces (see \citep[Thm. 18(4)]{Gio10e}). Property \eqref{eq:2_d_cohat},
for every $(s,u)\in S\times[0,1]$, yields the existence of
\begin{align*}
p_{s,u} & \in\FR^{{\sf p}_{s,u}}\\
\mathcal{U}_{s,u} & \text{ Fermat open neighbourhood of }p_{s,u}\text{ defined by }U_{s,u}\\
\mathcal{V}_{s,u} & \text{ Fermat open neighbourhood of }(s,u)\text{ defined by \ensuremath{V_{s,u}}}\text{ in }S\times[0,1]\\
\alpha_{s,u} & \in\mathcal{C}^{\infty}(U_{s,u}\times V_{s,u},\R)
\end{align*}
such that $d^{\vee}|_{\mathcal{V}_{s,u}}=\ext{\alpha_{s,u}}(p_{s,u},-)|_{\mathcal{V}_{s,u}}$.
Now we fix $s\in S$, and everything above only depends on $u\in[0,1]$.
Since $(s,u)\in\mathcal{V}_{u}=\ext{V_{u}}\cap\left(S\times[0,1]\right)$
and $V_{u}$ is open in $\R^{\sf s+1}$, we have that $(\st{s},\st{u})\in V_{u}$.
Thus, we may always assume that $V_{u}=A_{u}\times B_{u}$, for some
open subsets $A_{u}$ in $\R^{\sf s}$ and $B_{u}$ in $\R$, with
$(\st{s},\st{u})\in A_{u}\times B_{u}$. Hence, $(s,u)\in\ext{\left(A_{u}\times B_{u}\right)}=\ext{A_{u}}\times\ext{B_{u}}=\ext{V_{u}}$.
In this way, we obtain an open cover $\left(B_{u}\right)_{u\in[0,1]_{\R}}$
of $[0,1]_{\R}$, and from it we can extract a finite subcover. Therefore,
we can find $u_{1},\dots,u_{n}\in[0,1]_{\R}$ such that $\left(B_{u_{i}}\right)_{i=1}^{n}$
is an open cover of $[0,1]_{\R}$. As in Thm.~\ref{thm:existenceOfIntegralForFiniteBoundaries},
we may assume to have chosen the indices $i=1,\dots,n$ and have found
suitable $\delta_{i}$'s in $\R_{>0}$ so that
\begin{align*}
 & 0=u_{1}<u_{2}<\dots<u_{n}=1\\
 & B_{u_{i}}=(u_{i}-\delta_{i},u_{i}+\delta_{i})_{\R}\\
 & B_{u_{i}}\cap B_{u_{i+1}}=(u_{i+1}-\delta_{i+1},u_{i}+\delta_{i})_{\R}\neq\emptyset\quad\forall i=1,\dots,n-1.
\end{align*}
Finally, set $A:=A_{u_{1}}\cap\dots\cap A_{u_{n}}$, which is an open
neighbourhood of $\st{s}$, so that $s\in\ext{A}$. Take a point $\bar{u}_{i}\in(u_{i+1}-\delta_{i+1},u_{i}+\delta_{i})_{\R}$.
For every $x\in\ext{A}\cap S$, we have
\[
d^{\vee}(x,u)=\ext{\alpha_{u_{i}}}(p_{u_{i}},x,u)\quad\forall u\in\ext{B_{u_{i}}}\cap[0,1],
\]
and hence Lem.\ \ref{lem:IntegralAsfiniteSum} implies that
\begin{align*}
\int_{0}^{1}d^{\vee}(x,u)\diff{u} & =\int_{0}^{\bar{u}_{1}}\ext{\alpha_{u_{1}}}(p_{u_{1}},x,s)\diff{s}+\int_{\bar{u}_{1}}^{\bar{u}_{2}}\ext{\alpha_{u_{2}}}(p_{u_{2}},x,s)\diff{s}+\dots\\
 & \phantom{=}+\int_{\bar{u}_{n-1}}^{1}\ext{\alpha_{u_{n}}}(p_{u_{n}},x,s)\diff{s}=:\ext{\gamma}(\bar{p},x),
\end{align*}
where $\bar{p}:=\left(p_{u_{1}},\dots,p_{u_{n}}\right)\in\mathcal{U}_{u_{1}}\times\dots\times\mathcal{U}_{u_{n}}\subseteq\FR^{{\sf p}_{u_{1}}+\dots+{\sf p}_{u_{n}}}$.
From the classical version of the differentiation under the integral
sign, we have that $\gamma\in\mathcal{C}^{\infty}(U\times A,\R)$,
where $U:=U_{u_{1}}\times\dots\times U_{u_{n}}$, and hence we get
the conclusion.\end{proof}
\begin{cor}
\label{cor:smoothnessOfIntegralOperators}Let $J$ be a non-infinitesimal
interval of $\FR$ . Then the function
\[
(f,u,v)\in\FDiff(J,\FR)\times J^{2}\mapsto\int_{u}^{v}f\in\FR
\]
is smooth.
\end{cor}
We can now prove the differentiation under the integral sign:
\begin{lem}
\label{lem:derivationUnderIntegralSign}Let $J$ be a non-infinitesimal
interval of $\FR$, and let $f:[0,1]\times J\ra\FR$ be a smooth function.
Then
\[
\frac{\diff{}}{\diff{x}}\left(\int_{0}^{1}f(s,x)\diff{s}\right)=\int_{0}^{1}\frac{\partial f}{\partial x}(s,x)\diff{s}\quad\forall x\in J.
\]
\end{lem}
\begin{proof}
\noindent From Cartesian closedness, it follows that the map $x\in J\mapsto f(-,x)\in\FR^{[0,1]}$
is smooth. Together with Lem.\ \ref{lem:smoothIntegralOperatorOn_0_1},
we know that 
\[
x\in J\mapsto\int_{0}^{1}f(s,x)\diff{s}\in\FR
\]
is smooth. Let $\partial_{2}f[-,-;-]$ be the incremental ratio of
$f$ with respect to its second variable. Then for each $x\in\text{int(}J)$
and for $h\in\FR$ sufficiently small, we have
\begin{align*}
\int_{0}^{1}f(s,x+h)\diff{s} & =\int_{0}^{1}\left\{ f(s,x)+h\cdot\partial_{2}f[s,x;h]\right\} \diff{s}\\
 & =\int_{0}^{1}f(s,x)\diff{s}+h\cdot\int_{0}^{1}\partial_{2}f[s,x;h]\diff{s}.
\end{align*}
Therefore, the uniqueness of the incremental ratio $R[-,-]$ of the
function $\int_{0}^{1}f(s,-)\diff{s}$ yields that
\[
R[x,h]=\int_{0}^{1}\partial_{2}f[s,x;h]\diff{s}\quad\forall^{0}h.
\]
Hence,
\begin{align*}
\frac{\diff{}}{\diff{x}}\left(\int_{0}^{1}f(s,x)\diff{s}\right) & =R[x,0]\\
 & =\int_{0}^{1}\partial_{2}f[s,x;0]\diff{s}\\
 & =\int_{0}^{1}\frac{\partial f}{\partial x}(s,x)\diff{s},
\end{align*}
as expected.
\end{proof}
The following lemma is needed for the ``global'' proof of the additivity
of integral with respect to the integration intervals:
\begin{lem}
\label{lem:differentiationIntegral-x-Under}Let $J$ be a non-infinitesimal
interval of $\FR$, let $f:J\ra\FR$ be a smooth function, and let
$v\in J$. Then
\[
\frac{\diff{}}{\diff{x}}\left(\int_{x}^{v}f\right)=-f(x).
\]
\end{lem}
\begin{proof}
Set $\phi(s):=x+(v-x)s$ for each $s\in[0,1]$. From Thm.~\ref{thm:integrationByChangeOfVariable}
(integration by substitution), we get
\[
\int_{x}^{v}f(t)\diff{t}=\int_{0}^{1}f\left[x+(v-x)s\right]\cdot(v-x)\diff{s}.
\]
From this, using Lem.\ \ref{lem:derivationUnderIntegralSign} (differentiation
under the integral sign), the product rule (\citep[Thm. 28]{Gio10e})
and the chain rule (\citep[Thm. 29]{Gio10e}), we get
\begin{align}
\frac{\diff{}}{\diff{x}}\left(\int_{x}^{v}f\right) & =-\int_{0}^{1}f\left[x+(v-x)s\right]\diff{s}+\nonumber \\
 & \phantom{=}+(v-x)\int_{0}^{1}f'\left[x+(v-x)s\right]\cdot(1-s)\diff{s}\nonumber \\
 & =-\int_{0}^{1}f\left[x+(v-x)s\right]\diff{s}+\nonumber \\
 & \phantom{=}+\int_{x}^{v}f'(t)\diff{t}-\int_{0}^{1}s\cdot f'\left[\phi(s)\right]\cdot\phi'(s)\diff{s}.\label{eq:1-derivInt-x-under}
\end{align}
Integration by parts (Thm.~\ref{thm:classicalFormulasOfIntegralCalculus}\ref{enu:intByParts})
applied to the last integral yields
\begin{align*}
\int_{0}^{1}s\cdot f'\left[\phi(s)\right]\cdot\phi'(s)\diff{s} & =\left[(f\circ\phi)(s)\cdot s\right]_{0}^{1}-\int_{0}^{1}(f\circ\phi)(s)\diff{s}\\
 & =f(v)-\int_{0}^{1}f\left[x+(v-x)s\right]\diff{s}.
\end{align*}
Substituting in \eqref{eq:1-derivInt-x-under} and using Thm.~\ref{thm:classicalFormulasOfIntegralCalculus}\ref{enu:foundamental},
we obtain
\begin{align*}
\frac{\diff{}}{\diff{x}}\left(\int_{x}^{v}f\right) & =-\int_{0}^{1}f\left[x+(v-x)s\right]\diff{s}+f(v)-f(x)-f(v)+\\
 & \phantom{=}+\int_{0}^{1}f\left[x+(v-x)s\right]\diff{s}\\
 & =-f(x),
\end{align*}
as desired.
\end{proof}
Now we are led to the ``global'' proof of the additivity of integral
with respect to the integration intervals:
\begin{cor}
\label{cor:integralAdditiveDomain}Let $J$ be a non-infinitesimal
interval of $\FR$, let $f:J\ra\FR$ be a smooth function, and let
$u,$ $v$, $w\in J$. Then we have:
\[
\int_{u}^{v}f+\int_{v}^{w}f=\int_{u}^{w}f.
\]
\end{cor}
\begin{proof}
Set $F(x):=\int_{u}^{x}f+\int_{x}^{w}f$ for all $x\in J$. Then Lem.\ \ref{lem:differentiationIntegral-x-Under}
yields that $F'(x)=f(x)-f(x)=0$ at each $x\in J$. Hence by Lem.~\ref{lem:intervalsConnected},
$F$ is constant and $F(v)=F(u)$, which is our conclusion since $\int_{u}^{u}f=0$.
\end{proof}
Using Cor.~\ref{cor:integralAdditiveDomain}, we can calculate the
incremental ratio of the integral function $\int_{u}^{(-)}f$ and
get a form of mean value theorem called Hadamard's lemma as follows:
\begin{cor}
\label{cor:incrementalRatioIntegralFunction}Let $J$ be a non-infinitesimal
interval of $\FR$, let $f:J\ra\FR$ be a smooth function, and let
$u\in J$. Then
\begin{enumerate}[leftmargin=*,label=(\roman*),align=left ]
\item \label{enu:int_x_x+h}for each $x\in J$ and $h\in\FR$, if $x+h\in J$,
then $\int_{x}^{x+h}f(s)\diff{s}=h\cdot\int_{0}^{1}f(x+hs)\diff{s}$.
Therefore, the incremental ratio of the smooth function $\int_{u}^{(-)}f$
is
\[
\left(\int_{u}^{(-)}f\right)'[x,h]=\int_{0}^{1}f(x+hs)\diff{s}\quad\forall x\in\text{\emph{int}}(J)\,\forall^{0}h\in\FR;
\]

\item \label{enu:Hadamard}$\forall x\in\text{\emph{int}}(J)\,\forall^{0}h\in\FR:\ f(x+h)-f(x)=h\cdot\int_{0}^{1}f'(x+hs)\diff{s}$. 
\end{enumerate}
\end{cor}
\begin{proof}
\ref{enu:int_x_x+h}: The first statement follows from Thm.~\ref{thm:integrationByChangeOfVariable},
and the second statement then follows from Cor.\ \ref{cor:integralAdditiveDomain}
and Thm.\ \ref{thm:Fermat-Reyes-h-small}.

\noindent \ref{enu:Hadamard}: It follows by applying \ref{enu:int_x_x+h}
to the smooth function $f':\text{int}(J)\ra\FR$ together with Thm.~\ref{thm:classicalFormulasOfIntegralCalculus}\ref{enu:foundamental}.
\end{proof}

\subsection{Standard and infinitesimal parts of an integral}

To find a useful formula for the standard and the infinitesimal parts
of an integral, we introduce the following:
\begin{defn}
\label{def:stdInfParts}Let $S\subseteq\FR^{n}$ and let $f:S\ra\FR^{d}$
be a smooth function. Then
\begin{enumerate}[leftmargin=*,label=(\roman*),align=left ]
\item $\st{S}:=\left\{ \st{x}\in\R^{n}\mid x\in S\right\} $.
\item If $\st{S}\subseteq S$, then we set $\st{\!f}:r\in\st{S}\mapsto\st{\left[f(r)\right]}\in\R^{d}$,
which is called the \emph{standard part} of $f$.
\item Assume that $\st{S}$ is an open subset of $\R^{n}$, and $\ext{\left(\st{S}\right)}=S$.
Then we set $\sh{\!f}:=\ext{\left(\st{\!f}\right)}$ and $\delta f:=f-\sh{\!f}$,
which are called, respectively, the \emph{shadow} and the \emph{infinitesimal}
part of the smooth function $f$. Analogously, we set $\delta x:=x-\st{x}$
for any $x\in\FR$. Note that $\sh{\!f}(r)=\st{\!f}(r)$ for each
$r\in\st{S}$.
\end{enumerate}
\end{defn}
\noindent For example, if $S:=\text{int}(a,b)$, then $\st{S}=(\st{a},\st{b})_{\R}\subseteq S$,
$\st{S}$ is an open subset of $\R$, and $\ext{\left(\st{S}\right)}=S$.
Note that the standard part $\st{\!f}$ is a function defined and
with values in standard points, whereas the shadow $\sh{\!f}$ can
have nonstandard values for nonstandard points in the domain. Therefore,
$(\delta f)(x)$ is not necessarily equal to $\delta(f(x))$, unless
$x\in\st{S}\subseteq S$; see \ref{enu:stdFunctionAtStdNum} of the
following lemma. The correctness of the definition of shadow is proved
in \ref{enu:shadowAndInfPartAreSmooth} of the following lemma:
\begin{lem}
\label{lem:stdPartAndShadowOfFunction}Let $S\subseteq\FR^{n}$. Assume
that $\st{S}\subseteq S$, and let $f:S\ra\FR^{d}$ be a smooth function.
Then
\begin{enumerate}[leftmargin=*,label=(\roman*),align=left ]
\item \label{enu:stdFunctionAtStdNum}$\left(\st{\!f}\right)(\st{x})=\st{\left[f(x)\right]}\quad\forall x\in S$.
\item \label{enu:stdFcnIsSmooth}If $\st{S}$ is an open subset of $\R^{n}$,
then $\st{\!f}\in\Coo(\st{S},\R^{d})$.
\item \label{enu:shadowAndInfPartAreSmooth}If $\st{S}$ is an open subset
of $\R^{n}$ and $\ext{\left(\st{S}\right)}=S$, then $\sh{\!f}\in\FDiff(S,\FR^{d})$
and $\delta f\in\FDiff(S,D_{\infty}^{d})$.
\item \label{enu:stdAndInfPartsIntegral}Under the assumptions of \ref{enu:shadowAndInfPartAreSmooth},
if we further assume that $n=1$, $S\supseteq\text{\emph{int}}(a,b)$
and $u$, $v\in\text{\emph{int}}(a,b)$, then the standard part of
the integral $\int_{u}^{v}f$ is given by
\[
\st{\!\left(\int_{u}^{v}f\right)}=\Rint_{\st{u}}^{\st{v}}\st{\!f}=\int_{\st{u}}^{\st{v}}\sh{\!f},
\]

and the infinitesimal part is given by
\begin{align*}
\delta\left(\int_{u}^{v}f\right) & =\int_{u}^{\st{u}}\sh{\!f}+\int_{\st{v}}^{v}\sh{\!f}+\int_{u}^{v}\delta f\\
 & =\sum_{i=1}^{[\omega(v)]}\frac{\sh{\!f}^{(i-1)}(\st{v})}{i!}\left(\delta v\right)^{i}-\sum_{j=1}^{[\omega(u)]}\frac{\sh{\!f}^{(j-1)}(\st{u})}{j!}\left(\delta u\right)^{j}+\int_{u}^{v}\delta f\\
 & \in D_{\infty},
\end{align*}
 where $[\omega(u)]$ denotes the integer part of $\omega(u)$.

\end{enumerate}

As trivial consequences of \ref{enu:stdAndInfPartsIntegral}, any
integral over an infinitesimal integration interval is infinitesimal,
and any integral of an infinitesimal-valued smooth function is infinitesimal.

\end{lem}
\begin{proof}
\ref{enu:stdFunctionAtStdNum}: We can write $f(y)=\ext{\alpha}(p,y)$
for each $y\in\ext{V}\cap S$, where $\st{x}\in V$ and $V$ is an
open subset of $\R^{n}$. Therefore, $\left(\st{\!f}\right)(\st{x})=\st{\left[f(\st{x})\right]}=\st{\left[\ext{\alpha}(p,\st{x})\right]}=\alpha(\st{p},\st{x})$.
On the other side, $\st{\left[f(x)\right]}=\st{\left[\ext{\alpha}(p,x)\right]}=\alpha(\st{p},\st{x})$.

\noindent \ref{enu:stdFcnIsSmooth}: We write again $f|_{\ext{V}\cap S}=\ext{\alpha}(p,-)|_{\ext{V}\cap S}$,
where now $V$ is an open neighborhood of $r\in\st{S}\subseteq S$.
But $V\cap\st{S}\subseteq\ext{V}\cap S$, so $\st{\!f}|_{V\cap\st{\!S}}=\alpha(\st{p},-)|_{V\cap\st{\!S}}$,
$V\cap\st{S}$ is an open neighborhood of $r$ and $\alpha(\st{p},-)|_{V\cap\st{S}}$
is a smooth function. This proves our claim.

\noindent \ref{enu:shadowAndInfPartAreSmooth}: This is clear from
the property of the extension $\ext{(-)}$ of ordinary smooth functions
and from the assumption $\ext{\left(\st{S}\right)}=S$. Observe that,
using the previous notations, $\sh{\!f}(x)=\ext{\left(\st{\!f}\right)}(x)=\ext{\alpha}(\st{p},x)$
for $x\in S$ (with $\alpha$ and $p$ depending on $x$). Hence,
$(\delta f)(x)=f(x)-\sh{\,f}(x)=\ext{\alpha}(p,x)-\ext{\alpha}(\st{p},x)\in D_{\infty}^{d}$.

\noindent \ref{enu:stdAndInfPartsIntegral}: Lem.\ \ref{lem:IntegralAsfiniteSum}
yields
\[
\st{\,\left(\int_{u}^{v}f\right)}=\st{\left[\int_{u}^{\bar{x}_{1}}\ext{\alpha_{1}}(p_{1},s)\diff{s}+\ldots+\int_{\bar{x}_{n-1}}^{v}\ext{\alpha_{n}}(p_{n},s)\diff{s}\right]}.
\]
The inside of the brackets of the right hand side of this equality
is a smooth function of $u$, $\bar{x}_{1},\ldots,\bar{x}_{n-1}$,
$v$, $p_{1},\ldots,p_{n}$. Taking the standard parts (that, as usual,
includes the use of the dominated convergence theorem for $t\to0^{+}$
in $\R$) and using \eqref{eq:RiemannInt}, we obtain that
\begin{align*}
\st{\,\left(\int_{u}^{v}f\right)} & =\int_{\st{u}}^{\bar{x}_{1}}\alpha_{1}(\st{p_{1}},s)\diff{s}+\ldots+\int_{\bar{x}_{n-1}}^{\st{v}}\alpha_{n}(\st{p_{n}},s)\diff{s}\\
 & =\int_{\st{u}}^{\st{v}}\sh{\!f}\\
 & =\Rint_{\st{u}}^{\bar{x}_{1}}\alpha_{1}(\st{p_{1}},s)\diff{s}+\ldots+\Rint_{\bar{x}_{n-1}}^{\st{v}}\alpha_{n}(\st{p_{n}},s)\diff{s}\\
 & =\Rint_{\st{u}}^{\st{v}}\st{\!f}.
\end{align*}
Now, we can compute the infinitesimal part by
\begin{align*}
\delta\left(\int_{u}^{v}f\right) & =\int_{u}^{v}f-\int_{\st{u}}^{\st{v}}\sh{\!f}\\
 & =\int_{u}^{\st{u}}(\sh{\!f}+\delta f)+\int_{\st{u}}^{\st{v}}(\sh{\!f}+\delta f)+\int_{\st{v}}^{v}(\sh{\!f}+\delta f)-\int_{\st{u}}^{\st{v}}\sh{\!f}\\
 & =\int_{\st{u}+\delta u}^{\st{u}}\sh{\!f}+\int_{\st{v}}^{\st{v}+\delta v}\sh{\!f}+\int_{u}^{v}\delta f\\
 & =-\sum_{j=1}^{[\omega(u)]}\frac{\sh{\!f}^{(j-1)}(\st{u})}{j!}\left(\delta u\right)^{j}+\sum_{i=1}^{[\omega(v)]}\frac{\sh{\!f}^{(i-1)}(\st{v})}{i!}\left(\delta v\right)^{i}+\int_{u}^{v}\delta f,
\end{align*}
where, in the last step, we have used the infinitesimal Taylor's formula
\eqref{eq:TaylorStd} for (the extension of) ordinary smooth functions.
\end{proof}
The standard and infinitesimal parts of general Fermat spaces will
be discussed in \citep[Sec. 4.2]{GioWu15b}.

\subsection*{Example: Divergence and curl}

Now we use suitable infinitesimals in Fermat reals to revisit the
classical concepts of divergence and curl.

Traditionally, $\text{div}\vec{A}(x)$ is understood as the density
of the flux of a vector field $\vec{A}\in\Coo(U,\R^{3})$ through
an ``infinitesimal parallelepiped'' centered at the point $x$ in
an open set $U\subseteq\R^{3}$. To formalize this concept, we take
three vectors $\vec{h}_{1}$, $\vec{h}_{2}$, $\vec{h}_{3}\in\FR^{3}$
and express them with respect to a fixed base $\vec{e}_{1}$, $\vec{e}_{2}$,
$\vec{e}_{3}\in\R^{3}$ as 
\[
\vec{h}_{i}=k_{i}^{1}\cdot\vec{e}_{1}+k_{i}^{2}\cdot\vec{e}_{2}+k_{i}^{3}\cdot\vec{e}_{3},\quad\text{where }k_{i}^{j}\in\FR.
\]
We say that $P:=(x,\vec{h}_{1},\vec{h}_{2},\vec{h}_{3})$ is a \emph{(third
order) infinitesimal parallelepiped} if
\[
x\in\R^{3}\text{ and each }k_{i}^{j}\in D_{3}.
\]
Here the requirement that each $k_{i}^{j}\in D_{3}$ is to make sure
that the multiplication of any four (and hence more) such $k_{i}^{j}$'s
is zero. The flux of the vector field $\vec{A}$ through such a parallelepiped
(toward the outer) is, by definition, the sum of the fluxes through
every ``face'': 
\begin{align*}
\int_{P}\vec{A}\boldsymbol{\cdot}\vec{n}\,\diff{S} & :=\int_{-\frac{1}{2}}^{\frac{1}{2}}\diff{t}\int_{-\frac{1}{2}}^{\frac{1}{2}}\ext{\vec{A}}(x-\frac{1}{2}\vec{h}_{3}+t\vec{h}_{1}+s\vec{h}_{2})\boldsymbol{\cdot}\vec{(h}_{2}\times\vec{h}_{1})\,\diff{s}+\\
{} & +\int_{-\frac{1}{2}}^{\frac{1}{2}}\diff{t}\int_{-\frac{1}{2}}^{\frac{1}{2}}\ext{\vec{A}}(x+\frac{1}{2}\vec{h}_{3}+t\vec{h}_{1}+s\vec{h}_{2})\boldsymbol{\cdot}(\vec{h}_{1}\times\vec{h}_{2})\,\diff{s}+\cdots,
\end{align*}
where the fluxes through the two opposite faces spanned by $\vec{h}_{1}$
and $\vec{h}_{2}$ are explicit in the above formula, $\times$ is
the cross product, and the dots $\cdots$ indicate similar terms for
the other faces of the parallelepiped. Let us note that e.g.\ the
function $s\mapsto\ext{\vec{A}}(x-\frac{1}{2}\vec{h}_{3}+t\vec{h}_{1}+s\vec{h}_{2})$
is a quasi-standard smooth function. In this case, the fixed parameter
is $p=(x,t,\vec{h}_{1},\vec{h}_{2},\vec{h}_{3})\in\FR^{13}$. It is
easy to prove that if $\vec{A}\in\Coo(U,\R^{3})$ and the oriented
volume ${\rm Vol}(\vec{h}_{1},\vec{h}_{2},\vec{h}_{3})$ of the infinitesimal
parallelepiped $P=(x,\vec{h}_{1},\vec{h}_{2},\vec{h}_{3})$, is not
zero, then the following ratio between infinitesimals exists and is
independent of $\vec{h}_{1}$, $\vec{h}_{2}$, $\vec{h}_{3}$ (see
\citep{Gio10a} for the notion of ratio between infinitesimals): 
\[
{\rm div}\vec{A}(x):=\frac{1}{{\rm Vol}(\vec{h}_{1},\vec{h}_{2},\vec{h}_{3})}\cdot\int_{P}\vec{A}\boldsymbol{\cdot}\vec{n}\,\diff{S}\in\R.
\]
To define the curl of a vector field $\vec{A}\in\Coo(U,\R^{3})$,
we say that $C:=(x,\vec{h}_{1},\vec{h}_{2})$ is a \emph{(second order)
infinitesimal cycle }if 
\[
x\in U\text{ and each }k_{i}^{j}\in D_{2}.
\]
Here the requirement that each $k_{i}^{j}\in D_{2}$ is to make sure
that the multiplication of any three (and hence more) such $k_{i}^{j}$'s
is zero. The circulation of the vector field $\vec{A}$ on this cycle
$C$ is defined as the sum of the ``line integrals'' over each ``side'':
\[
\int_{C}\vec{A}\boldsymbol{\cdot}\vec{t}\,\diff{l}:=\int_{-\frac{1}{2}}^{\frac{1}{2}}\ext{\vec{A}}(x-\frac{1}{2}\vec{h}_{2}+t\vec{h}_{1})\boldsymbol{\cdot}\vec{h}_{1}\,\diff{t}-\int_{-\frac{1}{2}}^{\frac{1}{2}}\ext{\vec{A}}(x+\frac{1}{2}\vec{h}_{2}+t\vec{h}_{1})\boldsymbol{\cdot}\vec{h}_{1}\,\diff{t}+\cdots,
\]
where the line integrals of the two opposite sides spanned by $\vec{h}_{1}$
are explicit in the above formula, and the dots $\cdots$ indicate
similar terms for the other sides of the cycle $C$ (noting the consistency
of the orientation on each side). Once again, using exactly the calculations
frequently done in elementary courses of physics, one can prove that
there exists one and only one vector, $\text{{\rm curl}}\,\vec{A}(x)\in\R^{3}$,
such that 
\[
\int_{C}\vec{A}\boldsymbol{\cdot}\vec{t}\,\diff{l}=\text{{\rm curl}}\vec{A}(x)\boldsymbol{\cdot}(\vec{h}_{1}\times\vec{h}_{2})
\]
for every infinitesimal cycle $C=(x,\vec{h}_{1},\vec{h}_{2})$, representing
thus the (vector) density of the circulation of $\vec{A}$.

\section{\label{sec:Inequalities-for-integrals}Inequalities for integrals}

In this section, we derive similar inequalities for integrals of smooth
functions to the classical ones for integrals of ordinary smooth functions.
Thanks to the total order on $\FR$, we are able to prove the monotonicity
of integrals, a replacement of the inequality for absolute value of
smooth functions, and the Cauchy-Schwarz inequality, with respect
to arbitrary integration interval. The strategy is, we first prove
inequality for real integration interval, and then use integration
by substitution for arbitrary integration interval. 

Recall (see e.g.\ \citep{Gio10b}) that if $x$, $y\in\FR$, then
$x\le y$ if and only if for any $x_{t},y_{t}\in\R_{o}[t]$ with $x=[x_{t}]$
and $y=[y_{t}]$, there exists $z_{t}\in\R_{o}[t]$ such that $[z_{t}]=0$
in $\FR$ and $x_{t}\le y_{t}+z_{t}$ for $t\in\R_{\ge0}$ sufficiently
small. In proving these integral inequalities, we will use \citep[Thm. 28]{Gio10b}:
if $(x_{t})_{t}$ is any little-oh polynomial representing $x\in\FR$,
then $x\ge0$ is equivalent to $x_{t}\ge0$ for $t$ sufficiently
small, or $x=0$ in $\FR$ (i.e.,~$x_{t}=o(t)$ as $t\to0^{+}$ in
$\R$).

The first two results concerns the monotonicity of integrals:
\begin{lem}
\label{lem:monotoneIntegral}Let $J$ be a non-infinitesimal interval
of $\FR$, and let $f$, $g:J\ra\FR$ be smooth functions. Let $a$,
$b\in J\cap\R$, with $a<b$. Then the assumption
\begin{equation}
g(x)\le f(x)\quad\forall x\in[a,b]_{\R}\label{eq:g-le-f}
\end{equation}
implies $\int_{a}^{b}g\le\int_{a}^{b}f$.\end{lem}
\begin{proof}
It suffices to prove the conclusion for $g=0$. By Lem.~\ref{lem:IntegralAsfiniteSum},
we can write 
\[
\int_{a}^{b}f=\sum_{i=1}^{n}\int_{\bar{x}_{i-1}}^{\bar{x}_{i}}\ext{\alpha_{i}}(p_{i},s)\diff{s},
\]
 where $\alpha_{i}\in\Coo(U_{i}\times V_{i},\R)$, $U_{i}\subseteq\R^{{\sf p}_{i}}$,
$p_{i}\in\ext{U_{i}}$, $V_{i}=(x_{i}-\delta_{i},x_{i}+\delta_{i})_{\R}$,
$\delta_{i}\in\R_{>0}$, $\bar{x}_{i}\in(x_{i+1}-\delta_{i+1},x_{i}+\delta_{i})_{\R}=V_{i}\cap V_{i+1}\subset(x_{i},x_{i+1})_{\R}$,
and
\begin{gather}
f(x)=\ext{\alpha_{i}}(p_{i},x)\quad\forall x\in\ext{V_{i}}\cap J,\label{eq:f-alpha-V_i}\\
a=\bar{x}_{0}=x_{1}<x_{2}<\ldots<x_{n}=\bar{x}_{n}=b.\nonumber 
\end{gather}
For each $i=1,\ldots,n$ and each $s\in[\bar{x}_{i-1},\bar{x}_{i}]_{\R}$,
we have $a=\bar{x}_{0}\le\bar{x}_{i-1}\le s\le\bar{x}_{i}\le\bar{x}_{n}=b$,
so $s\in[a,b]_{\R}$. Moreover, $\bar{x}_{i-1}\in V_{i}$ and $\bar{x}_{i}\in V_{i}$
yield that $x_{i}-\delta_{i}<\bar{x}_{i-1}\le s\le\bar{x}_{i}<x_{i}+\delta_{i}$,
so $s\in V_{i}\cap J$. Thereby, assumption \eqref{eq:g-le-f} and
equality \eqref{eq:f-alpha-V_i} imply that $f(s)=\ext{\alpha_{i}}(p_{i},s)\ge0$.
The mentioned result \citep[Thm.\ 28]{Gio10b} yields that
\begin{gather}
\forall^{0}t\ge0:\ \alpha_{i}((p_{i})_{t},s)\ge0\nonumber \\
\text{or}\label{eq:alternOrder}\\
\ext{\alpha_{i}}(p_{i},s)=0\text{ in }\FR.\nonumber 
\end{gather}
Set $w_{i}(t,s):=0$ if $\alpha_{i}((p_{i})_{t},s)\ge0$ and $w_{i}(t,s):=\alpha_{i}((p_{i})_{t},s)$
otherwise. Then,
\begin{equation}
\forall t\in\R_{\ge0}\,\forall s\in[\bar{x}_{i-1},\bar{x}_{i}]_{\R}:\ \alpha_{i}((p_{i})_{t},s)\ge w_{i}(t,s),\label{eq:alpha-ge-w}
\end{equation}
and from \eqref{eq:alternOrder} we get that $w_{i}(t,s)=o(t)$ as
$t\ra0^{+}$ in $\R$. Note also that $w_{i}(t,-):[\bar{x}_{i-1},\bar{x}_{i}]_{\R}\ra\R$
is integrable by definition, so that \eqref{eq:alpha-ge-w} implies
\begin{equation}
\Rint_{\bar{x}_{i-1}}^{\bar{x}_{i}}\alpha_{i}((p_{i})_{t},s)\diff{s}\ge\Rint_{\bar{x}_{i-1}}^{\bar{x}_{i}}w_{i}(t,s)\diff{s}\quad\forall t\in\R_{\ge0}.\label{eq:inequality-alpha-w}
\end{equation}
But Lebesgue dominated convergence theorem and the fact that $w_{i}(t,s)=o(t)$
yield
\begin{equation}
\lim_{t\ra0^{+}}\frac{1}{t}\cdot\Rint_{\bar{x}_{i-1}}^{\bar{x}_{i}}w_{i}(t,s)\diff{s}=\Rint_{\bar{x}_{i-1}}^{\bar{x}_{i}}\lim_{t\ra0^{+}}\frac{w_{i}(t,s)}{t}\diff{s}=0,\label{eq:WIsZero}
\end{equation}
i.e., the right hand side of the inequality \eqref{eq:inequality-alpha-w}
viewed as a function of $t$ is $o(t)$ as $t\ra0^{+}$ in $\R$,
and hence it is in $\R_{o}[t]$ and it represents $0$ in $\FR$.
In other words, we have proved that
\[
\int_{\bar{x}_{i-1}}^{\bar{x}_{i}}\ext{\alpha}_{i}(p_{i},s)\diff{s}\geq0.
\]
 We thus have
\begin{align*}
\int_{a}^{b}f & =\sum_{i=1}^{n}\int_{\bar{x}_{i-1}}^{\bar{x}_{i}}\ext{\alpha}_{i}(p_{i},s)\diff{s}\geq0,
\end{align*}
which proves our conclusion. We finally note that in \eqref{eq:WIsZero},
all the endpoints of the integration intervals, which include $a=\bar{x}_{0}$
and $b=\bar{x}_{n}$, do not depend on $t$.
\end{proof}
We can now extend this result to the case where the endpoints $a$,
$b$ of the integration interval are arbitrary Fermat reals.
\begin{thm}
\label{thm:monotoneIntegral}Let $J$ be a non-infinitesimal interval
of $\FR$, and let $f$, $g:J\ra\FR$ be smooth functions. Let $a$,
$b\in J$, with $a<b$. Then the assumption
\begin{equation}
g(x)\le f(x)\quad\forall x\in[a,b]\label{eq:g-le-f-FR}
\end{equation}
implies $\int_{a}^{b}g\le\int_{a}^{b}f$.
\end{thm}
Note that in this theorem (also the next two theorems), we include
the case when the integration interval is infinitesimal, i.e., $\st{a}=\st{b}$.
\begin{proof}
As above, without loss of generality, we may assume $g=0$. Using
integration by substitution (Thm.~\ref{thm:integrationByChangeOfVariable}),
we can write
\begin{equation}
\int_{a}^{b}f=(b-a)\cdot\int_{0}^{1}f[a+(b-a)\sigma]\diff{\sigma}\label{eq:inequality-substitution}
\end{equation}
For each $\sigma\in[0,1]_{\R}$, we have $a\leq a+(b-a)\sigma\leq b$.
Note that $a+(b-a)\sigma$ is in general not standard if $a$ is not
standard, which justifies the stronger assumption \eqref{eq:g-le-f-FR}
compared to the previous \eqref{eq:g-le-f}. Assumption \eqref{eq:g-le-f-FR}
yields that $f\left[a+(b-a)\sigma\right]\ge0$. Lem.~\ref{lem:monotoneIntegral}
implies that the right hand side of the integral in \eqref{eq:inequality-substitution}
is non-negative, from which the conclusion follows.
\end{proof}
Although we haven't justified the validity of integrals of absolute
values of smooth functions, the following result is a possible substitute
of the property $\left|\int_{a}^{b}f\right|\le\int_{a}^{b}|f|$ for
classical integrals. 
\begin{thm}
\label{thm:substIntAbsValue}Let $J$ be a non-infinitesimal interval
of $\FR$, and let $f$, $g:J\ra\FR$ be smooth functions. Let $a$,
$b\in J$, with $a<b$. Then the assumption
\begin{equation}
|f(x)|\le g(x)\quad\forall x\in[a,b]\label{eq:abs-f-le-g}
\end{equation}
implies $\left|\int_{a}^{b}f\right|\le\int_{a}^{b}g$.\end{thm}
\begin{proof}
The proof is conceptually identical to those of Lem.~\ref{lem:monotoneIntegral}
and Thm.~\ref{thm:monotoneIntegral}. We first prove the conclusion
for $a$, $b\in\R$. The interval $[a,b]_{\R}$ can be covered by
a finite number of intervals $X_{i}:=(x_{i}-\delta_{i},x_{i}+\delta_{i})_{\R}$,
where we can write both $f(x)=\ext{\alpha}_{i}(p_{i},x)$ and $g(x)=\ext{\beta}_{i}(q_{i},x)$
for every $x\in\ext{X_{i}}\cap[a,b]$. Furthermore, we have
\begin{align*}
\int_{a}^{b}f & =\sum_{i=1}^{n}\int_{\bar{x}_{i-1}}^{\bar{x}_{i}}\ext{\alpha}_{i}(p_{i},s)\diff{s},\\
\int_{a}^{b}g & =\sum_{i=1}^{n}\int_{\bar{x}_{i-1}}^{\bar{x}_{i}}\ext{\beta}_{i}(q_{i},s)\diff{s},
\end{align*}
where $\bar{x}_{i}\in(x_{i+1}-\delta_{i+1},x_{i}+\delta_{i})_{\R}$
and $a=\bar{x}_{0}=x_{1}<x_{2}<\ldots<x_{n}=\bar{x}_{n}=b$. For each
$i=1,\ldots,n$ and each $s\in[\bar{x}_{i-1},\bar{x}_{i}]_{\R}$,
assumption \eqref{eq:abs-f-le-g} and \citep[Thm.\ 28]{Gio10b} yield
that
\begin{gather*}
\forall^{0}t\in\R_{\ge0}:\ |\alpha_{i}((p_{i})_{t},s)|\le\beta_{i}((q_{i})_{t},s)\\
\text{or}\\
|\ext{\alpha}_{i}(p_{i},s)|=\ext{\beta}_{i}(q_{i},s)\text{ in }\FR.
\end{gather*}
Set
\[
w_{i}(t,s):=\begin{cases}
0, & \text{if }|\alpha_{i}((p_{i})_{t},s)|\le\beta_{i}((q_{i})_{t},s)\\
|\alpha_{i}((p_{i})_{t},s)|-\beta_{i}((q_{i})_{t},s), & \text{otherwise.}
\end{cases}
\]
Therefore,
\begin{equation}
\forall t\in\R_{\ge0}\,\forall s\in[\bar{x}_{i-1},\bar{x}_{i}]_{\R}:\ \left|\alpha_{i}((p_{i})_{t},s)\right|\le\beta_{i}((q_{i})_{t},s)+w_{i}(t,s).\label{eq:abs-alpha-le-beta+w}
\end{equation}
As in the proof of Lem.~\ref{lem:monotoneIntegral}, one can show
that for any fixed $s\in[\bar{x}_{i-1},\bar{x}_{i}]_{\R}$, $w_{i}(t,s)=o(t)$,
and for any fixed $t\in\R_{\geq0}$, the function $w_{i}(t,-):[\bar{x}_{i-1},\bar{x}_{i}]\ra\R$
is integrable. Again, one gets $\int_{\bar{x}_{i-1}}^{\bar{x}_{i}}w_{i}(t,s)\diff{s}=o(t)$,
and as a consequence, $\int_{\bar{x}_{i-1}}^{\bar{x}_{i}}w_{i}(t,s)\diff{s}\in\R_{o}[t]$
and it represents $0$ in $\FR$. From inequality \eqref{eq:abs-alpha-le-beta+w}
we get the following inequality of little-oh polynomials:
\begin{align*}
\left|\sum_{i=1}^{n}\Rint_{\bar{x}_{i-1}}^{\bar{x}_{i}}\alpha_{i}((p_{i})_{t},s)\diff{s}\right| & \leq\sum_{i=1}^{n}\Rint_{\bar{x}_{i-1}}^{\bar{x}_{i}}\left|\alpha_{i}((p_{i})_{t},s)\right|\diff{s}\\
 & \le\sum_{i=1}^{n}\Rint_{\bar{x}_{i-1}}^{\bar{x}_{i}}\beta_{i}((q_{i})_{t},s)\diff{s}+\sum_{i=1}^{n}\Rint_{\bar{x}_{i-1}}^{\bar{x}_{i}}w_{i}(t,s)\diff{s}.
\end{align*}
Since the last summand in the above formula represents zero in $\FR$,
this proves our conclusion for $a$, $b\in\R$.

Now, we can proceed as in the proof of Thm.~\ref{thm:monotoneIntegral}:
\begin{align*}
\left|\int_{a}^{b}f\right| & =(b-a)\cdot\left|\int_{0}^{1}f\left[a+(b-a)\sigma\right]\diff{\sigma}\right|\\
 & \le(b-a)\cdot\int_{0}^{1}g\left[a+(b-a)\sigma\right]\diff{\sigma}\\
 & =\int_{a}^{b}g.
\end{align*}

\end{proof}
We have just showed that Thm.~\ref{thm:substIntAbsValue} follows
from Lem.~\ref{lem:IntegralAsfiniteSum} and the analogous property
of the Riemann integral. In the same way, we can prove the following:
\begin{thm}[Cauchy-Schwarz inequality]
\label{thm:Cauchy-Schwarz}Let $J$ be a non-infinitesimal interval
of $\FR$, and let $f$, $g:J\ra\FR$ be smooth functions. Let $a$,
$b\in J$, with $a<b$. Then $(\int_{a}^{b}fg)^{2}\le(\int_{a}^{b}f^{2})\cdot(\int_{a}^{b}g^{2})$.
\end{thm}

\section{\label{sec:Multiple-integrals-and}Multiple integrals and Fubini's
theorem}

In previous sections, we have developed integral calculus with an
interval as integration domain. In this section we will define the
integral (and multiple integral) of a smooth function with a more
general integration domain. The approach is similar to that used for
the Peano-Jordan measure. 

We denote by $\mathscr{I}$ the set of all the intervals of $\bar{\FR}$
(see Sec.~\ref{sec:Background-for-FR}). Thanks to the total order
relation $\le$ on $\FR$, every interval uniquely determines its
endpoints. (We recall that this property is not true e.g. in SDG,
where $\le$ is a pre-order.) Since the domain of every smooth function
defined on an interval of $\FR$ can always be extended to a Fermat
open subset (Lem.~\ref{lem:extensionFcnOutsideBoundaries}), we can
therefore define the integral over an interval as follows:
\begin{defn}
\label{def:nonOrientedInt}Let $J$ be a non-infinitesimal interval,
let $f:J\ra\FR$ be a smooth function, and let $I\in\mathscr{I}$
be an interval contained in $J$. We define
\[
\int_{I}f:=\int_{\inf(I)}^{\sup(I)}f.
\]

\end{defn}
The following lemma permits to order the endpoints of pairwise disjoint
intervals.
\begin{lem}
\label{lem:disjIntervals}If $I$, $J\in\mathscr{I}$ are disjoint
intervals, then either $\sup(I)\le\inf(J)$ or $\sup(J)\le\inf(I)$.\end{lem}
\begin{proof}
Set $a:=\inf(I)$, $b:=\sup(I)$, $\bar{a}:=\inf(J)$ and $\bar{b}:=\sup(J)$.
By contradiction, we assume $\bar{a}<b$ and $a<\bar{b}$. We distinguish
four cases: $a\le\bar{a}$ and $\bar{b}\le b$, $a\le\bar{a}$ and
$b\le\bar{b}$, $\bar{a}\le a$ and $\bar{b}\le b$, $\bar{a}\le a$
and $b\le\bar{b}$. In each case, we can find a point in $I\cap J$.
\end{proof}
This lemma permits to integrate smooth functions over a finite pairwise
disjoint union of intervals:
\begin{thm}
\label{thm:intOnSumIntervals}Let $\left\{ I_{1},\ldots,I_{n}\right\} $
and $\left\{ J_{1},\ldots,J_{m}\right\} $ be two finite pairwise
disjoint families of $\mathscr{I}$ such that
\[
\bigcup_{i=1}^{n}I_{i}=\bigcup_{j=1}^{m}J_{j}=:U.
\]
Let $J$ be a non-infinitesimal interval and let $f:J\ra\FR$ be a
smooth function. Then
\begin{equation}
\sum_{i=1}^{n}\int_{I_{i}}f=\sum_{j=1}^{m}\int_{J_{j}}f=:\int_{U}f.\label{eq:intSumIntervals}
\end{equation}
\end{thm}
\begin{proof}
Set $a_{i}:=\inf(I_{i})$, $b_{i}:=\sup(I_{i})$. Thanks to Lem.~\ref{lem:disjIntervals},
we can always rearrange the indices $i$, $j$ so that
\[
a_{1}\le b_{1}\le\ldots\le a_{n}\le b_{n}.
\]
Since the intervals are pairwise disjoint, an equality of the type
$b_{i}=a_{i+1}$ may happen only if $b_{i}\notin I_{i}$ or $a_{i+1}\notin I_{i+1}$.
If $I_{i}\cup I_{i+1}$ is connected, we could replace $I_{i}$ and
$I_{i+1}$ in the set $\{I_{1},\ldots,I_{n}\}$ by $I_{i}\cup I_{i+1}$
to get another finite pairwise disjoint family $\{I_{1}',\ldots,I_{n-1}'\}$
of $\mathscr{I}$. In this case, because $\int_{I_{i}}f+\int_{I_{i+1}}f=\int_{a_{i}}^{b_{i+1}}f=\int_{I_{i}\cup I_{i+1}}f$,
we know that 
\[
\sum_{i=1}^{n}\int_{I_{i}}f=\sum_{i=1}^{n-1}\int_{I_{i}'}f.
\]
In other words, we can combine two ``joinable'' intervals to get
a new set of finite pairwise disjoint family, and keep the integral
unchanged. After finitely many such procedures, we reach a set of
finite pairwise disjoint family consisting of all the connected components
of $U$, and by the invariance of integrals, we get
\[
\sum_{i=1}^{n}\int_{I_{i}}f=\sum\left\{ \int_{I}f\mid I\text{ is a connected component of }U\right\} .
\]
We arrive at the same result by considering $\sum_{j=1}^{m}\int_{J_{j}}f$.
\end{proof}
In the following definition, we identify the general domains for integrals
of higher dimensions.
\begin{defn}
\label{def:boxAndElementarySets}A\emph{ box} in $\FR^{d}$ is a set
of the form $B=\prod_{i=1}^{d}I_{i}$, where $I_{i}\in\mathscr{I}$
are intervals. We denote by $\mathscr{B}^{d}$ the set of all the
boxes in $\FR^{d}$. An \emph{elementary set} $U\in\mathscr{E}$ is
a set of the form $\bigcup_{i=1}^{n}B_{i}$, where $B_{i}\in\mathscr{B}^{d}$
are boxes.
\end{defn}
\ 
\begin{thm*}
\label{thm:ElementaryAreBoole}$\mathscr{E}$ is a Boolean algebra
with respect to union, intersection and set difference.\end{thm*}
\begin{proof}
It is enough to show that $\mathscr{E}$ is closed under these operations.
The closure of $\mathscr{E}$ with respect to union is clear. If $U=\bigcup_{i=1}^{n}B_{i}$
and $V=\bigcup_{k=1}^{m}C_{k}$, where $B_{i}=\prod_{j=1}^{d}I_{ij}$
and $C_{k}=\prod_{j=1}^{d}J_{kj}$ are boxes, then
\[
U\cap V=\bigcup\left\{ \prod_{j=1}^{d}I_{ij}\cap J_{kj}\mid i=1,\ldots,n\ ,\ j=1,\ldots,m\right\} .
\]
Since $I_{ij}\cap J_{kj}\in\mathscr{I}$ (including the empty interval),
$U\cap V\in\mathscr{E}$. To prove the closure of $\mathscr{E}$ with
respect to set difference, we first introduce the notion of complement
$A^{c}:=\FR^{d}\setminus A$ for any subset $A\subseteq\FR^{d}$.
Using the same notations, since
\[
C_{k}^{c}=\left(\prod_{j=1}^{d}J_{kj}\right)^{c}=\cup\left\{ \prod_{i=1}^{l-1}\FR\times(J_{kl}^{c})^{s}\times\prod_{i=l+1}^{d}\FR\mid l=1,\ldots,d\;,\;s=\pm\right\} \in\mathscr{E},
\]
where $(J_{kl}^{c})^{+}=\{x\in\FR\mid x>y\;\forall y\in J_{kl}\}$
and $(J_{kl}^{c})^{-}=\{x\in\FR\mid x<y\;\forall y\in J_{kl}\}$.
we know that $V^{c}=(\cup_{k=1}^{m}C_{k})^{c}=\cap_{k=1}^{m}C_{k}^{c}\in\mathscr{E}$,
and hence $U\setminus V=U\cap V^{c}\in\mathscr{E}$. 
\end{proof}
In our framework, the following result corresponds to Fubini's theorem.
\begin{thm}
\label{thm:Fubini}Let $B=\prod_{i=1}^{d}I_{i}\in\mathscr{B}^{d}$
be a box in $\FR^{d}$. Let $U\subseteq\FR^{d}$ be a Fermat open
set such that $B\subseteq U$, and let $f:U\ra\FR$ be a smooth function.
Let $(\sigma_{1},\ldots,\sigma_{d})$ be a permutation of $\{1,\ldots,d\}$.
Then
\begin{multline*}
\int_{I_{1}}\ldots\left(\int_{I_{d}}f(x_{1},\ldots,x_{d})\diff{x_{d}}\right)\ldots\diff{x_{1}}\\
=\int_{I_{\sigma_{1}}}\ldots\left(\int_{I_{\sigma_{d}}}f(x_{1},\ldots,x_{d})\diff{x_{\sigma_{d}}}\right)\ldots\diff{x_{\sigma_{1}}}.
\end{multline*}
We denote this integral by $\int_{B}f$.\end{thm}
\begin{proof}
We prove only for the case $d=2$, since the general case is similar.
Let $a$, $b$, $\bar{a}$, $\bar{b}$ be the endpoints of $I_{1}$
and $I_{2}$, respectively. Set $F(t):=\int_{\bar{a}}^{\bar{b}}f(t,y)\diff{y}$
and $G(t):=\int_{\bar{a}}^{\bar{b}}\left(\int_{a}^{t}f(x,y)\diff{x}\right)\diff{y}$
for each $t\in I_{1}$. Both $F$ and G are smooth functions such
that $G(a)=0$ and
\[
G'(t)=\int_{\bar{a}}^{\bar{b}}\frac{\partial}{\partial t}\left(\int_{a}^{t}f(x,y)\diff{x}\right)\diff{y}=\int_{\bar{a}}^{\bar{b}}f(t,y)\diff{y}=F(t)\quad\forall t\in I_{1}.
\]
Therefore, $G=\int_{a}^{(-)}F$ and $G(b)=\int_{\bar{a}}^{\bar{b}}\left(\int_{a}^{b}f(x,y)\diff{x}\right)\diff{y}=\int_{a}^{b}F(t)\diff{t}=\int_{a}^{b}\left(\int_{\bar{a}}^{\bar{b}}f(x,y)\diff{y}\right)\diff{x}$,
which is our conclusion.
\end{proof}
\noindent Similarly to Thm.~\ref{thm:intOnSumIntervals}, we can
finally prove the following:
\begin{thm}
\label{thm:intElemSets}Let $\left\{ B_{1},\ldots,B_{n}\right\} $
and $\left\{ C_{1},\ldots,C_{m}\right\} $ be two finite pairwise
disjoint families of $\mathscr{B}^{d}$ such that
\[
\bigcup_{i=1}^{n}B_{i}=\bigcup_{j=1}^{m}C_{j}=:U.
\]
Let $V$ be a Fermat open subset of $\FR^{d}$ containing $U$, and
let $f:V\ra\FR$ be a smooth \emph{function}. Then
\begin{equation}
\sum_{i=1}^{n}\int_{B_{i}}f=\sum_{j=1}^{m}\int_{C_{j}}f=:\int_{U}f.\label{eq:intSumIntervals-1}
\end{equation}

\end{thm}

\section{Conclusions}

The main aim of the present work is to lay the foundation of the integral
calculus for smooth functions defined on the ring of Fermat reals.
In the Introduction \ref{sec:Introduction}, we have already summarized
the major achievements of this theory compared to other non-Archimedean
theories. We can outline them by saying: better order properties and
easier to deal with infinite dimensional integral operators.

The present work confirms that the ring of Fermat reals does not represent
a new foundation of the calculus. On the other hand, the definition
of integral as primitive (Def.~\ref{def:integral}) and the corresponding
simple ``global proofs'' (see e.g.~Thm.~\ref{thm:classicalFormulasOfIntegralCalculus})
can suggest ideas for a new approach to a subset of the classical
integral calculus. On these bases, we can also think of the possibility
to approach the integral calculus of smooth functions from an axiomatic
point of view.

Finally, our work shows that all the classical instruments to deal
with integrals of smooth functions with infinitesimal parameters are
now available. This is important for the applications in Physics.

A suitable notion of convergence for sequences of smooth functions,
the exchange between limit and integral, and a more general class
of integration domains remain open. We plan to explore them in a future
work.

\end{document}